\documentclass[11pt]{article}
\usepackage[T1]{fontenc}
\usepackage{amsfonts}
\usepackage{amsmath}
\usepackage{amssymb}
\usepackage{amsthm}
\usepackage{bbm}
\usepackage{bm}
\usepackage{mathrsfs}
\usepackage{verbatim}
\usepackage{setspace}
\usepackage{color}
\usepackage{pdfsync}
\usepackage{enumitem}
\usepackage{graphicx}
\usepackage{subfigure}
\usepackage{tikz}
\usetikzlibrary{patterns}
\usepackage[total={360pt, 561pt}, voffset=20pt]{geometry}  %
\usepackage{layout}
\usepackage[pdfborder={0 0 0}]{hyperref}
\hypersetup{
  urlcolor = black,
  pdfauthor = {Marcel Nutz},
  pdfkeywords = {Optimal Transport, Quadratic Regularization, Potentials},
  pdftitle = {Quadratically Regularized Optimal Transport: Existence and Multiplicity of Potentials},
  pdfsubject = {Quadratically Regularized Optimal Transport: Existence and Multiplicity of Potentials},
  pdfpagemode = UseNone
}

\usepackage[capitalize]{cleveref} %

\theoremstyle{plain}
\newtheorem{theorem}{Theorem}[section]
\newtheorem{proposition}[theorem]{Proposition}
\newtheorem{lemma}[theorem]{Lemma}

\theoremstyle{definition}
\newtheorem{definition}[theorem]{Definition}
\newtheorem{remark}[theorem]{Remark}

\newtheorem{example}[theorem]{Example}

\theoremstyle{remark}

\renewenvironment{thebibliography}[1]{%
\begin{oldthebibliography}{#1}%
\setlength{\baselineskip}{.9em}
\linespread{.93}
\small
\setlength{\parskip}{.25ex}%
\setlength{\itemsep}{.20em}%
}%
{%
\end{oldthebibliography}%
}
\newcommand{\eps}{\varepsilon}
\newcommand{\tmu}{\tilde{\mu}}
\newcommand{\tnu}{\tilde{\nu}}

\newcommand{\N}{\mathbb{N}}

\newcommand{\R}{\mathbb{R}}

\newcommand{\X}{\mathsf{X}}
\newcommand{\Y}{\mathsf{Y}}

\newcommand{\cB}{\mathcal{B}}
\newcommand{\cC}{\mathcal{C}}

\newcommand{\cF}{\mathcal{F}}

\newcommand{\cL}{\mathcal{L}}

\newcommand{\cP}{\mathcal{P}}

\newcommand{\cZ}{\mathcal{Z}}

\newcommand{\sD}{\mathscr{D}}

\newcommand{\sP}{\mathscr{P}}

\DeclareMathOperator{\QOT}{QOT}
\DeclareMathOperator{\osc}{osc}
\DeclareMathOperator{\dom}{dom}

\DeclareMathOperator{\proj}{proj}

\DeclareMathOperator{\spt}{spt}
\DeclareMathOperator{\Int}{int}

\DeclareMathOperator*{\argmin}{arg\, min}

\newcommand{\as}{\mbox{-a.s.}}

\newcommand{\1}{\mathbf{1}}

\newcommand{\qforallq}{\quad\mbox{for all}\quad}
\newcommand{\qforq}{\quad\mbox{for}\quad}

\newcommand{\qandq}{\quad\mbox{and}\quad}
\newcommand{\qonq}{\quad\mbox{on}\quad}

\newcommand{\mykill}[1]{}
\numberwithin{equation}{section}

\begin{document}

\title{\vspace{-1em}
Quadratically Regularized Optimal Transport: Existence and Multiplicity of Potentials}
\author{
  Marcel Nutz%
  \thanks{
  Columbia University, Depts.\ of Statistics and Mathematics, mnutz@columbia.edu. Research supported by NSF Grants DMS-1812661, DMS-2106056, DMS-2407074. The author is grateful to Alberto Gonz\'{a}lez-Sanz, Gilles Mordant, Andr\'{e}s Riveros Valdevenito and Johannes Wiesel for stimulating discussions and to two anonymous referees for their detailed and constructive comments.}
  }
\date{\today}
\maketitle 

\begin{abstract}
  The optimal transport problem with quadratic regularization is useful when sparse couplings are desired. The density of the optimal coupling is described by two functions called potentials; equivalently, potentials can be defined as a solution of the dual problem. We prove the existence of potentials for a general square-integrable cost. Potentials are not necessarily unique, a phenomenon directly related to sparsity of the optimal support. For discrete problems, we describe the family of all potentials based on the connected components of the support, for a graph-theoretic notion of connectedness. On the other hand, we show that continuous problems have unique potentials under standard regularity assumptions, regardless of sparsity. Using potentials, we prove that the optimal support is indeed sparse for small regularization parameter in a continuous setting with quadratic cost, which seems to be the first theoretical guarantee for  sparsity in this context.
\end{abstract}

\vspace{.9em}

{\small

}
\vspace{.9em}

\section{Introduction}\label{se:intro}

We are concerned with quadratically regularized optimal transport; that is,
\begin{align}\label{eq:QOTintro}
  \QOT_{\eps}(\mu,\nu) := 
  \inf_{\pi\in\Pi(\mu,\nu)} \int_{\X\times\Y} c(x,y)\, \pi(dx,dy) + \frac{\eps}{2} \left\|\frac{d\pi}{dP}\right\|^{2}_{L^{2}(P)}
\end{align}
where $\mu,\nu$ are given marginal distributions on separable spaces $\X,\Y$ and $\Pi(\mu,\nu)$ denotes the set of their couplings. Moreover, $P$ is a (product) reference measure for computing the density $d\pi/dP$. The function $c:\X\times\Y\to\R$ is a given ``cost'' and the regularization parameter~$\eps>0$ controls the strength of the regularization. Quadratic regularization is also called Euclidean or $\chi^{2}$ regularization. We shall prove for a general cost $c\in L^{2}(P)$ that the unique optimal coupling $\pi_{*}$ for~\eqref{eq:QOTintro} has a density of the form 
\begin{align}\label{eq:potentialsIntro}
\frac{d\pi_{*}}{dP}(x,y)=\big(f(x)+g(y)-c(x,y)/\eps\big)_{+}
\end{align}
for certain functions $f:\X\to\R$ and $g:\Y\to\R$; cf.\ \cref{th:main}. These functions are called \emph{potentials}. They are also described as solutions to a dual problem, but are in general non-unique. For discrete problems, non-uniqueness is the typical case, and we shall describe the family of all potentials (\cref{th:multiplicityDiscrete}). Our description uses a decomposition of the support of $\pi_{*}$ into ``components'' (for a certain notion of connectedness; cf.\ \cref{de:connected}). In particular, the multiplicity of the potentials is directly related to sparsity of the support---sparsity typically implies that there are multiple components, and hence potentials. For continuous and semi-discrete problems, the situation is quite different: under standard regularity conditions we show that the potentials are unique (up to an additive constant); cf.\ \cref{th:dualUniqueness}.
We use potentials to show that the support of $\pi_{\eps}$ is sparse for small $\eps>0$, in the continuous setting with quadratic Euclidean cost~$c$. Specifically, we show that the support is contained in a neighborhood of the graph of Brenier's map (\cref{th:sparse}). To the best of our knowledge, this is the first theoretical result on sparsity of continuous regularized transport. We mention the recent follow-up works \cite{GonzalezSanzNutz.24b, WieselXu.24} discussing quantitative versions of \cref{th:sparse}.

Regularization has many purposes in optimal transport---to facilitate computation, to obtain smoother couplings and dual potentials, to improve sampling complexity, and others. Two regularizations are primarily used: entropic regularization penalizes couplings by the Kullback--Leibler divergence while quadratic regularization penalizes by the $L^{2}$-norm of the density. Entropic regularization (EOT) is the most frequent choice, as it allows for Sinkhorn's algorithm (e.g., \cite{Cuturi.13, PeyreCuturi.19}) and has strong smoothness properties. %
This smoothness is intimately linked to the full support property of the optimal coupling, which can be a blessing or a curse (``overspreading'') depending on the application. While for small values of the regularization parameter~$\eps$, the actual weight of the EOT coupling might be quite small in large regions, the issue is aggravated by a second issue of EOT: its computation is difficult for small values of~$\eps$  (e.g., \cite{Schmitzer.19}). By contrast, quadratic regularization is empirically known to give rise to couplings with \emph{sparse support} for a range of~$\eps$ (e.g., \cite{blondel18quadratic}). Moreover, as its computation does not involve logarithms and exponentials, one can use regularization parameters that are several orders of magnitude smaller than in EOT without running into issues with machine precision. For those reasons, quadratic regularization is used in applications where sparsity and/or weak regularization are desired.

Quadratically regularized optimal transport was first addressed by~\cite{blondel18quadratic,EssidSolomon.18} in discrete settings. It is also a special case of optimal transport with convex regularization~\cite{DesseinPapadakisRouas.18}; see also the predecessors referenced in~\cite{EssidSolomon.18}. The formulation of~\cite{blondel18quadratic} is closer to ours; the authors present several experiments highlighting the sparsity of the optimal coupling and derive theoretical results regarding duality and convergence as the regularization parameter $\eps$ tends to zero. The authors further illustrate how entropic regularization can lead to blurrier results in image processing tasks.  
In \cite{EssidSolomon.18}, quadratic regularization is studied for a minimum-cost flow problem on a graph; this includes discrete optimal transport as a particular case.  
The authors introduce a Newton-type algorithm and discuss sparsity in several examples. In a continuous setting, several works including \cite{EcksteinKupper.21, GeneveyEtAl.16, GulrajaniAhmedArjovskyDumoulinCourville.17, LiGenevayYurochkinSolomon.20, seguy2018large} have applied optimization techniques on the dual problem of regularized optimal transport. For instance, \cite{LiGenevayYurochkinSolomon.20} applies neural networks and gradient descent to compute regularized Wasserstein barycenters. 
The authors compare entropic and quadratic regularization and 
highlight that the entropic penalty produces a blurrier image 
at the smallest computationally feasible regularization ($\eps=10^{-2}$ for entropic, $\eps=10^{-5}$ for quadratic). Recently, \cite{ZhangMordantMatsumotoSchiebinger.23} uses quadratically regularized optimal transport in a manifold learning task related to single cell RNA sequencing; specifically, the optimal coupling is used to produce an adaptive affinity matrix. In this context, sparsity is crucial to avoid bias---a full support coupling would introduce shortcuts through ambient space instead of following the data manifold. In this application, the transport problem is of ``self-transport'' type: the marginals $\mu=\nu$ are identical and the cost~$c$ is symmetric. In that situation we will show that the potentials $f,g$ can be chosen to be symmetric; i.e., $f=g$. While symmetry eliminates certain degrees of freedom, we shall see that non-uniqueness can still occur.

The first work rigorously addressing a continuous setting is~\cite{LorenzMannsMeyer.21}. The authors derive duality results and present two algorithms, a nonlinear Gauss--Seidel method and a semismooth Newton method. The theoretical results assume that the marginal spaces $\X,\Y$ are compact subsets of $\R^{d}$ and that the marginal distributions $\mu,\nu$ are absolutely continuous with densities uniformly bounded away from zero. The reference measure~$P$ is taken to be the Lebesgue measure. The authors apply weak* compactness in the space of Radon measures to solve the dual problem, and then this solution provides potentials. By contrast, our approach is not of topological nature. It covers in a unified way discrete and continuous settings, and different reference measures~$P$, avoiding technical restrictions almost entirely. The paper \cite{LorenzMahler.22} generalizes some of the results of~\cite{LorenzMannsMeyer.21} to Orlicz space regularizations and shows Gamma convergence as $\eps\to0$ to the unregularized optimal transport problem. 
This convergence is studied quantitatively in~\cite{EcksteinNutz.22}, where a rate of convergence is derived based on quantization arguments. The recent work~\cite{GarrizMolinaGonzalezSanzMordant.24} finds the optimal leading constant for that rate using connections with the porous medium equation. Meanwhile, \cite{BayraktarEcksteinZhang.22} shows stability with respect to the marginal distributions. Both~\cite{EcksteinNutz.22} and~\cite{BayraktarEcksteinZhang.22} cover quadratic regularization as a special case of more general $f$-divergences. 
The unpublished work \cite{DiMarinoGerolin.20b} also considers optimal transport with regularization by an $f$-divergence, with quadratic regularization being a special case. The authors emphasize the analogy to $c$-convex conjugation in optimal transport (cf.\ the semi-smooth dual studied, e.g., in~\cite{blondel18quadratic}) and use it to derive a priori estimates for the potentials. In a setting with uniformly bounded cost~$c$ and $P=\mu\otimes\nu$, these results are leveraged to obtain existence of the potentials and convergence of the nonlinear Gauss--Seidel algorithm. The paper also states results regarding the uniqueness of the potentials and differentiability of $\nu\mapsto \QOT_{\eps}(\mu,\nu)$ which however are  flawed. Specifically, uniqueness is asserted in a general setting including discrete problems, based on an assertion that the dual problem is strictly concave. We emphasize that the dual problem~\eqref{eq:dual} is not strictly concave in the case of quadratic regularization ($x\mapsto -x_{+}^{2}$ is constant on~$\R_{-}$) and uniqueness fails in simple situations such as 
$\mu=\nu=\frac12(\delta_{0}+\delta_{1})$ with $c(x,y)=|x-y|^{2}$ and $\eps=1/3$ (see \cref{ex:discrete} for details). In cases where uniqueness holds, it does so for very different reasons.

To the best of our knowledge, apart from the aforementioned, we are the first to describe the multiplicity of the potentials; specifically, to describe the family of all potentials in the discrete case and prove uniqueness in a continuous (and semi-continuous) case. The connection between the family of all potentials and ``components'' of the support also seems to be novel. While sparsity of the support has been highlighted as an empirical finding, \cref{th:sparse} seems to be the first theoretical result in its direction. As mentioned above, there is no analogue to this sparsity in EOT, where the support of the optimal coupling always equals the support of $\mu\otimes\nu$.

For the existence of the potentials, we pursue a novel path inspired by~\cite{Follmer.88}: while the works cited above attack the dual problem, we leverage the (straightforward) fact that the primal problem has a solution given by a Hilbert space projection. To construct potentials, we introduce approximating problems with finitely many equality constraints instead of the marginal constraints. Their solutions have the form~\eqref{eq:potentialsIntro} and converge to the optimal density. To achieve the passage to the limit, we must show that a sequence of functions $f_{n}(x)+g_{n}(y)$ converges to a limit of the form $f(x)+g(y)$. This problem is surprisingly subtle, but we can take advantage of insights found in the context of Schr\"odinger bridges \cite{FollmerGantert.97, RuschendorfThomsen.97}. This line of argument avoids the conditions on the marginals and costs in previous approaches. It also allows us to cover different reference measures~$P$ in a unified way. Once the form~\eqref{eq:potentialsIntro} is obtained, the properties of the dual problem follow easily by standard arguments.

The remainder of this paper is organized as follows. \Cref{se:main} states the existence  and duality results, and basic regularity properties of potentials. In \cref{se:multiplicity}, we characterize the family of all potentials in the discrete case and prove the uniqueness in the continuous case. \Cref{se:sparseQuadratic} applies potentials to prove sparsity  in the setting of quadratic cost. 
\Cref{se:proofMain} contains the proof of the existence and duality result, while \cref{se:proofMainSelf} proves the same result for the self-transport problem; i.e., with potentials $f=g$.

\section{Problem Formulation and Existence}\label{se:main}

Consider two Polish\footnote{More generally, our results hold for measurable spaces with countably generated $\sigma$-fields; cf.\ \cref{rk:minimalAssumptions}.} probability spaces $(\X,\cB(\X),\mu)$ and $(\Y,\cB(\Y),\nu)$. 
We endow $\X\times\Y$ with the product $\sigma$-field and denote by $\Pi(\mu,\nu)$ the set of couplings of $(\mu,\nu)$; that is, measures~$\pi$ on $\X\times\Y$ satisfying $\pi(A\times\Y)=\mu(A)$ for $A\in \cB(\X)$ and $\pi(\X\times B)=\nu(B)$ for $B\in \cB(\Y)$. We also use the standard notation $(f\oplus g)(x,y):=f(x)+g(y)$ for functions $f:\X\to\R$ and $g:\Y\to\R$, and $\cP(\X)$ for the set of probability measures on~$\X$.

We further consider measures $(\tmu,\tnu)\in\cP(\X)\times\cP(\Y)$ satisfying $\mu\sim\tmu$ and $\nu\sim\tnu$ (where $\sim$ denotes mutual absolute continuity) and
\begin{align}\label{eq:refMeasureCond}
  \frac{d\mu}{d\tmu}\in L^{2}(\tmu), \quad \left(\frac{d\mu}{d\tmu}\right)^{-1}\in L^{\infty}(\tmu), \quad \frac{d\nu}{d\tnu}\in L^{2}(\tnu), \quad \left(\frac{d\nu}{d\tnu}\right)^{-1}\in L^{\infty}(\tnu),
\end{align}
and denote their product
\begin{align}\label{eq:P}
P:=\tmu\otimes\tnu. 
\end{align} 
Finally, we are given a cost function 
\begin{align}\label{eq:cIntCond}
\begin{split}
 &c\in L^{2}(P) \quad\mbox{satisfying}\quad c\geq c_{1}\oplus c_{2} \\
 &\mbox{for some} \qquad c_{1}\in L^{1}(\mu)\cap L^{1}(\tmu), \quad c_{2}\in L^{1}(\nu)\cap L^{1}(\tnu); 
\end{split}
\end{align}
the lower bound ensures in particular that for any $\pi\in\Pi(\mu,\nu)$, the integral $\int c\,d\pi$ is well defined with values in $(-\infty,\infty]$. 
With this notation in place, the quadratically regularized optimal transport problem (with $\eps=1$) is
\begin{align}\label{eq:QOT}
  \inf_{\pi\in\Pi(\mu,\nu)} \int c\,d\pi + \frac12 \bigg\|\frac{d\pi}{dP}\bigg\|^{2}_{L^{2}(P)}
\end{align}
where (by convention) any coupling $\pi\not\ll P$ has infinite cost. The extension to a general regularization parameter $\eps>0$ is straightforward; see \cref{rk:regularizationParam}.

\begin{remark}\label{rk:refMeas}
(a) For the reference measure $P$ in~\eqref{eq:P}, our default is $(\tmu,\tnu):=(\mu,\nu)$. This choice leads to a meaningful problem~\eqref{eq:QOT} in discrete and continuous settings, and  convergence of solutions when continuous marginals are approximated by discrete ones. However, some works use other choices for $\tmu,\tnu$, especially uniform measures on a certain domain (usually discrete or in~$\R^{d}$). While the choice of reference measure is often not highlighted in the literature, we shall see that it can be quite crucial.\footnote{The function \texttt{ot.stochastic.loss{\_}dual{\_}quadratic} in the  \texttt{Python Optimal Transport} package follows \cite{seguy2018large} and assumes $(\tmu,\tnu)=(\mu,\nu)$; see \cite{POTquadratic}. The function \texttt{OptimalTransport.quadreg} in the \texttt{OptimalTransport.jl} package follows~\cite{LorenzMannsMeyer.21} and assumes that $\tmu,\tnu$ are uniform; see \cite{OTjuliaQuadratic}.}  For the entropic optimal transport problem, it is known that changing the measures $\tmu,\tnu$ does not affect the optimal coupling %
  (e.g., \cite{Nutz.20}). That fact has no analogue for the present problem. Consequently, we provide the results for general~$P$. %
\Cref{ex:zeroCostDiscrete} and \cref{pr:zeroCost} show that the optimal coupling and even the optimal support can depend on $\frac{d\mu}{d\tmu}$ and $\frac{d\nu}{d\tnu}$, even for the trivial cost $c\equiv0$.
  
  (b) The integrability condition on $(\frac{d\mu}{d\tmu})^{-1},(\frac{d\nu}{d\tnu})^{-1}$ in~\eqref{eq:refMeasureCond} is used only to obtain a convenient lower bound for potentials (\cref{le:lowerBounds}). No condition is needed for the existence results in \cref{pr:approxLimit,pr:approxLimitSelf}.
  
  (c) As usual in the literature on regularized optimal transport, we only consider~$P$ in product form, as that leads to dual optimality equations \eqref{eq:marginalEqnu}, \eqref{eq:marginalEqmu} of a form convenient for computations. While our problem~\eqref{eq:QOT} makes sense if merely $\mu\ll\tmu$ and $\nu\ll\tnu$, there is no real gain in generality for~\eqref{eq:QOT} from this weaker condition. Hence, we use equivalent measures. 

\end{remark}

In view of the second term in~\eqref{eq:QOT}, the minimization is equivalently restricted to couplings with square-integrable density. Indeed, let
\begin{align*}
  \cZ:=\{Z\in L^{2}(P):\, Z\,dP\in \Pi(\mu,\nu)\},
\end{align*}
which is nonempty as it contains $d(\mu\otimes\nu)/dP$ due to~\eqref{eq:refMeasureCond}.
Then we can rephrase~\eqref{eq:QOT} as
\begin{align}\label{eq:primal}
  \sP := \inf_{Z\in\cZ} \int cZ\,dP + \frac12 \|Z\|^{2}_{L^{2}(P)},
\end{align}
called the \emph{primal problem} below. Note that $\sP\in\R$  due to~\eqref{eq:cIntCond} and the Cauchy--Schwarz inequality.
The \emph{dual problem} is 
\begin{align}\label{eq:dual}
  \sD:=\sup_{f\in L^{1}(\mu),\, g\in L^{1}(\nu)} \int f\,d\mu + \int g\,d\nu - \frac12\int (f\oplus g - c)_{+}^{2}\,dP.
\end{align}
The two problems are in strong duality, and both admit optimizers.

\begin{theorem}\label{th:main}
  \begin{enumerate}
  \item 
  Strong duality holds: $\sP= \sD$.
  \item
  The primal problem~\eqref{eq:primal} has a unique solution $Z_{*}\in\cZ$  given by the $L^{2}(P)$-projection $Z_{*}=\argmin_{Z\in\cZ} \|Z+c\|_{L^{2}(P)}^{2}$ of $-c$ onto $\cZ$. In particular, $Z_{*}$ is characterized by  \[
    \langle Z_{*}+c,Z-Z_{*}\rangle_{L^{2}(P)}\geq0 \qforallq Z\in\cZ.
  \]
  The coupling $\pi_{*}\in\Pi(\mu,\nu)$ given by $d\pi_{*}=Z_{*}\,dP$ is the unique solution of the regularized transport problem~\eqref{eq:QOT}.

\item 
  There exist measurable functions $f:\X\to\R$ and $g:\Y\to\R$ %
  satisfying the following conditions, and these conditions are equivalent:
  \begin{enumerate}
  \item $(f\oplus g - c)_{+}$ is the density $Z_{*}\in\cZ$ of the optimal coupling~$\pi_{*}$,
  \item $(f\oplus g - c)_{+}$ is the density of some coupling $\pi\in\Pi(\mu,\nu)$,
  \item $(f,g)\in L^{1}(\mu)\times L^{1}(\nu)$ is a solution of the dual problem,
  \item $(f,g)$ satisfies the system
   \begin{align}
       \int_{\X} (f(x)+g(y)-c(x,y))_{+}\,\tmu(dx)&=\frac{d\nu}{d\tnu}(y) \quad \mbox{for $\tnu$-a.e.\ $y\in\Y$,} \label{eq:marginalEqnu}\\
       \int_{\Y} (f(x)+g(y)-c(x,y))_{+}\,\tnu(dy)&=\frac{d\mu}{d\tmu}(x) \quad \mbox{for $\tmu$-a.e.\ $x\in\X$.} \label{eq:marginalEqmu}
    \end{align}      
  \end{enumerate}
   Any such $(f,g)$, necessarily in $L^{1}(\mu)\times L^{1}(\nu)$, are called \emph{potentials}. 
   \item Suppose that $(\X,\mu,\tmu)=(\Y,\nu, \tnu)$ and $c(x,y)=c(y,x)$. Then the existence in~(iii) also holds with the additional requirement that $f=g$, and the dual problem~\eqref{eq:dual} has the same value~$\sD$ if restricted to $f=g$.
  \end{enumerate}
\end{theorem}

 For ease of reference, the following remark states the corresponding formulas for general regularization parameter $\eps>0$.

\begin{remark}[General $\eps>0$]\label{rk:regularizationParam}
Often the regularized transport problem is considered with a parameter $\eps>0$ for the quadratic penalty:
\begin{align*}
  \sP_{\eps} = \inf_{\pi\in\Pi(\mu,\nu)} \int c\,d\pi + \frac{\eps}{2} \|d\pi/dP\|^{2}_{L^{2}(P)}.
\end{align*}
The corresponding dual is
\begin{align}\label{eq:dualEps}
  \sD_{\eps} = \sup_{f\in L^{1}(\mu),\, g\in L^{1}(\nu)} \eps\int f\,d\mu + \eps\int g\,d\nu - \frac{\eps}{2}\int (f\oplus g - c/\eps)_{+}^{2}\,dP
\end{align}
and the optimal density then takes the form
$
  Z_{\eps} = (f\oplus g - c/\eps)_{+}^{2}
$
for the optimizers $(f,g)$ of~\eqref{eq:dualEps}. Note that $\sP_{\eps}=\eps \sP(c/\eps)=\eps \sD(c/\eps)=\sD_{\eps}$ if $\sP(\bar c)$ denotes the primal problem~\eqref{eq:primal} for the cost $\bar c$ and similarly $\sD(\bar c)$ for the dual. In situations where $\eps$ is varied, it is often convenient to consider the \emph{rescaled} potentials $(f_{\eps},g_{\eps}) := (\eps f,\eps g)$. After this change of variables, the dual problem reads
\begin{align}\label{eq:dualEpsRescaled}
  \sD_{\eps} = \sup_{\tilde f\in L^{1}(\mu),\, \tilde g\in L^{1}(\nu)} \int \tilde f\,d\mu + \int \tilde g\,d\nu - \frac{\eps}{2}\int \left(\frac{\tilde f\oplus \tilde g - c}{\eps}\right)_{+}^{2}\,dP
\end{align}
and the optimal density takes the form
\begin{align*}
  Z_{\eps} = \left(\frac{f_{\eps}\oplus g_{\eps} - c}{\eps}\right)_{+}
\end{align*}
for the optimizers $(f_{\eps},g_{\eps})$ of~\eqref{eq:dualEpsRescaled}. The rescaled potentials incorporate the correct scaling in particular for the limit $\eps\to0$. %
\end{remark}

For brevity, and without loss of generality, we use $\eps=1$  in the remainder of the paper, except in \cref{se:sparseQuadratic} where we consider the limit $\eps\to0$.

The proof of \cref{th:main} consists of two parts. The main part is to prove (iii)(a); i.e., that the optimal density $Z_{*}$ is of the form $(f\oplus g-c)_{+}$. The other assertions in (i)--(iii) follow from that fact and elementary arguments. To show the main part, we construct approximating problems whose solutions have the form $(f_{n}\oplus g_{n}-c)_{+}$ and converge to the optimal density $Z_{*}$ as $n\to\infty$. Then, we argue that $f_{n}\oplus g_{n}$ converges to a function $f\oplus g$ on a sufficiently large set. The details of the proof are deferred to \cref{se:proofMain}. In the symmetric setting of self-transport assumed in~(iv), the construction of \cref{se:proofMain} generally yields potentials $f\neq g$. \cref{se:proofMainSelf} presents a more refined construction guaranteeing $f=g$.

As we want to use~(d) in the derivation below, let us  observe that the equivalence of~(b) and~(d) in \cref{th:main} is straightforward and independent of all other claims.

\begin{proof}[Proof of (b)$\Leftrightarrow$(d)] 
The left-hand side in~\eqref{eq:marginalEqnu} is the formula for the density of the second marginal of the measure $d\pi:= (f\oplus g - c)_{+}\,dP$ wrt.\ $\tnu$. If $\pi\in\Pi(\mu,\nu)$, the second marginal is~$\nu$, giving the right-hand side. Similarly for~\eqref{eq:marginalEqmu}. Conversely, if the marginal densities equal $(d\mu/d\tmu,d\nu/d\tnu)$, then $\pi$ is a coupling.
\end{proof}

In the remainder of this section we gather some properties of potentials to be used in \cref{se:multiplicity}. (Additional bounds and integrability properties are stated in \cref{se:integrabilityProperties}.) 
Let $(f,g)$ be potentials. For $x\in\X$, consider
\begin{align*}%
    F_{x}(t):=\int_{\Y} (t+g(y)-c(x,y))_{+}\,\,\tnu(dy), \qquad t\in\R.
\end{align*} 
By~\eqref{eq:marginalEqmu}, there is a set $\X_{0}$ of full $\tmu$-measure such that for $x\in\X_{0}$,
  \begin{align}\label{eq:conjugateDefn}
    F_{x}(t)=\int_{\Y} (t+g(y)-c(x,y))_{+}\,\,\tnu(dy)=\frac{d\mu}{d\tmu}(x) \qforq t=f(x).
  \end{align} 
  As $c\in L^{2}(P)$ and $\mu\sim\tmu$, we may further choose $\X_{0}$ such that $c(x,\cdot)\in L^{1}(\tnu)$ and $\frac{d\mu}{d\tmu}(x)>0$ for $x\in\X_{0}$. We observe that for all $x\in\X_{0}$, the function $t\mapsto F_{x}(t)$ is continuous, nondecreasing, strictly increasing on the set where it is positive, $\lim_{t\to-\infty} F_{x}(t)=0$ and $\lim_{t\to\infty} F_{x}(t)=\infty$, by dominated/monotone convergence. For $x\in\X_{0}$, we conclude that there exists a unique $t$ with $F_{x}(t)=\frac{d\mu}{d\tmu}(x)$. In particular, the value of $f(x)$ is uniquely determined by $g$ and $c(x,\cdot)$, $\mu$-a.s. We record this fact for ease of reference.
  
\begin{lemma}\label{le:gGivenf}
  One potential uniquely determines the other: if $(f,g)$ and $(f',g)$ are potentials, then $f=f'$ $\mu$-a.s.
\end{lemma}

We have seen that for $x\in\X_{0}$, there exists a unique $t$ with $F_{x}(t)=\frac{d\mu}{d\tmu}(x)$. More generally, we may fix versions of $c$ and $d\mu/d\tmu$ such that this holds for all $x\in\X$. Then, the function $g^{c}(x):=\{t: F_{x}(t)=\frac{d\mu}{d\tmu}(x)\}$ is well-defined for all $x\in\X$. Moreover, $g^{c}(x)=f(x)$ for $x\in\X_{0}$ by~\eqref{eq:conjugateDefn}. In summary, $g^{c}$ is a version of the potential $f$ that is defined everywhere on $\X$. 
Following standard arguments in optimal transport, we may think of $g^{c}$ as a conjugate of~$g$. This point of view was emphasized in~\cite{DiMarinoGerolin.20b} to show how potentials inherit properties from the cost function. Specifically, the setting of~\cite{DiMarinoGerolin.20b} assumes that $(\tmu,\tnu)=(\mu,\nu)$. We state (a generalization of) that result in the next lemma, before discussing how it breaks down when $\mu\neq\tmu$.

\begin{lemma}[Oscillation]\label{le:unifCont}
  Suppose that $\mu=\tmu$ and define $$\Delta(x,x'):= \sup_{y\in \Y}|c(x,y)-c(x',y)|\in[0,\infty].$$ If $(f,g)$ are potentials, then $f$ satisfies $|f(x)-f(x')|\leq \Delta(x,x')$ for all $x,x'$ outside a $\mu$-nullset. In particular:
  \begin{enumerate}
  \item 
  If the oscillation $\osc c(\cdot,y)\leq C$ for all~$y$, then $\osc f\leq C$ $\mu$-a.s.\footnote{We use the notation $\osc \varphi := \sup\varphi-\inf\varphi$ for a real function $\varphi$.}
  
  \item 
  If $c$ is bounded, then $f$ is bounded. More precisely, $\|f+\alpha\|_{L^{\infty}(\mu)}\leq 2\|c\|_{\infty}$ after choosing the centering $\alpha$ such that $\|(f+\alpha)_{+}\|_{L^{\infty}(\mu)}=\|(f+\alpha)_{-}\|_{L^{\infty}(\mu)}$.
  
    \item 
  If a metric is given on~$\X$ and $x\mapsto c(x,y)$ is uniformly continuous with modulus of continuity~$\omega$, then~$f$ admits a version that is $\omega$-continuous on $\spt\mu$. If $x\mapsto c(x,y)$ is $L$-Lipschitz, then $f$ admits a version that is $L$-Lipschitz on~$\X$. 
  \end{enumerate}
\end{lemma}

\begin{proof}
  For all $x,x'$ in a set $\X_{0}$ of full $\mu$-measure, writing $\Delta=\Delta(x,x')$, we use~\eqref{eq:conjugateDefn} at~$x$ and $\frac{d\mu}{d\tmu}\equiv1$ to find  
    \begin{align*}
     \int(f(x) +g(y)-c(x',y)\mp\Delta)_{+}\,\tnu(dy) 
     &\lesseqgtr \int  (f(x)+g(y)-c(x,y))_{+}\,\tnu(dy) \\
     & = \frac{d\mu}{d\tmu}(x)=1=\frac{d\mu}{d\tmu}(x').
  \end{align*} 
  Now using~\eqref{eq:conjugateDefn} at~$x'$ and the strict monotonicity of $F_{x'}$ yield $$f(x')\in [f(x)-\Delta,f(x)+\Delta].$$ Hence $|f(x)-f(x')|\leq \Delta$ and then $\osc f \leq \sup_{x,x'} \Delta(x,x') \leq \sup_{y}\osc c(\cdot,y)$. This is (i) which in turn implies (ii). If $c$ is $\omega$-continuous, the above shows that $f$ is $\omega$-continuous on~$\X_{0}$. Thus $f$ can be uniquely extended to a $\omega$-continuous function on the closure of~$\X_{0}$, which is a superset of~$\spt\mu$ as $\mu(\X_{0})=1$. If $f$ is Lipschitz on $\spt\mu$, we can further use McShane's theorem and extend to a Lipschitz function on the whole space.
\end{proof}

\Cref{le:unifCont} showcases the idea that the potential~$f$ inherits regularity from the cost~$c$.\footnote{In contrast to EOT, this principle does not extend to higher-order regularity in general: the potential need not have a $C^{1}$ version even for a $C^{\infty}$ cost. For quadratic cost $c(x,y)=(x-y)^{2}$ on $\R$ with $\mu=\tilde\mu$ uniform on $[-1,1]$ and $\nu=\tilde\nu$ uniform on $[-2,-1]\cup[1,2]$, the derivative $f'$ exists and is continuous except at $x=0$, where $f'$ has a jump discontinuity. This follows from the formula for $f'$ given in \cite{GonzalezSanzNutz.24b}. In that sense, the Lipschitz result is optimal for general marginals.} This breaks down when $\mu\neq\tmu$. For instance, \cref{pr:zeroCost} shows in particular that $f=\frac{d\mu}{d\tmu}$ (up to additive constant) when $c\equiv 0$ and $\nu=\tnu$. Hence, in general, the regularity of the potentials depends on the regularity of the marginal densities. The next lemma exemplifies how~\eqref{eq:conjugateDefn} can still be used to obtain regularity results.

\begin{lemma}[Lipschitz potentials]\label{le:unifContDiffMarg}
  Let $\X$ be endowed with a metric. Let $c$ be bounded and let $x\mapsto c(x,y)$ be Lipschitz uniformly in~$y$. Moreover, let $\frac{d\mu}{d\tmu}$ and $(\frac{d\mu}{d\tmu})^{-1}$ be bounded and Lipschitz, and let $\frac{d\nu}{d\tnu}$ be bounded. Then $f$ admits a version that is bounded and Lipschitz 
  on~$\X$. 
\end{lemma} 

\begin{proof}
  For brevity, we write $\xi(x)=(\frac{d\mu}{d\tmu}(x))^{-1}$ and
  \[
     \tilde{f}(x) = \xi(x)f(x), \qquad \tilde{g}(x,y) = \xi(x)g(y),\qquad \tilde{c}(x,y)=\xi(x)c(x,y).
  \]
  As $c$ and $\frac{d\nu}{d\tnu}$ are bounded, \cref{le:lowerBounds} shows that $\|g\|_{\infty}<\infty$. 
  As~$\xi$ is Lipschitz and~$g$ is bounded, we see that $\tilde{g}(x,y)$ is Lipschitz in $x$, uniformly in $y$. Similarly, $x\mapsto \tilde{c}(x,y)$ is bounded and Lipschitz (uniformly in~$y$) as a product of bounded Lipschitz functions. Thus
  \[
    \Delta(x,x')
    := \sup_{y\in \Y}|\tilde{c}(x,y)-\tilde{c}(x',y)|+\sup_{y\in \Y}|\tilde{g}(x,y)-\tilde{g}(x',y)| \leq Ld_{\X}(x,x')
  \]  
  for some $L>0$, where $d_{\X}$ denotes the metric on $\X$. Multiplying~\eqref{eq:conjugateDefn} with $\xi(x)>0$ leads to
  \begin{align*}
    \int_{\Y} (\xi(x)t+\tilde{g}(x,y)-\tilde{c}(x,y))_{+}\,\,\tnu(dy)=1 \qforq t=f(x)
  \end{align*} 
  and then arguing  as in the proof of \cref{le:unifCont}  yields
  \[
    \tilde{f}(x')\in [\tilde{f}(x)-\Delta, \tilde{f}(x)+\Delta] \quad\qforq \Delta=\Delta(x,x');
  \]
  that is, 
  $|\tilde{f}(x)-\tilde{f}(x')|\leq Ld_{\X}(x,x')$. Thus $\tilde{f}$ is bounded Lipschitz. As $\xi^{-1}=\frac{d\mu}{d\tmu}$ is also bounded Lipschitz, the product $f=\xi^{-1}\tilde{f}$ is again bounded Lipschitz. It follows that $f$ admits a Lipschitz version on~$\X_{0}$ and hence on~$\spt\mu$. That version can again be extended to a Lipschitz function on~$\X$ by McShane's theorem.
\end{proof}

\section{Multiplicity of Potentials}\label{se:multiplicity}

In this section we study the multiplicity of the potentials $(f,g)$. %
There is always a trivial non-uniqueness, as $(f+\alpha,g-\alpha)$ have the same sum $(f+\alpha)\oplus(g-\alpha)=f\oplus g$ for any $\alpha\in\R$. The main question is whether the potentials are unique up to this additive constant, or if there are \emph{further} degrees of freedom in choosing the potentials.

\subsection{Discrete Case}\label{se:discrete}

We shall describe the full family of potentials based on the geometry of the support $\spt \pi_{*}$ of the optimal coupling. Uniqueness typically fails. Let us first study a minimal example to obtain some guidance.

\begin{example}\label{ex:discrete}
  Let $\X=\Y=\{0,1\}$ and $\mu=\tmu=\nu=\tnu=\frac12(\delta_{0}+\delta_{1})$. Let $c(x,y)=(2+\gamma)\1_{x\neq y}$; we first consider the case $\gamma\geq0$. We claim that the optimal density is
  \begin{align*}
    Z_{*}(x,y)=2\1_{x=y},
  \end{align*}
  meaning that the optimal coupling~$\pi_{*}$ is the uniform measure on the diagonal, and that $(f,g)$ are potentials if and only if
  \begin{align}\label{eq:discrExPot}
  \begin{cases}
    &f(0) = \alpha, \quad g(0) = 2-\alpha, \\
    &f(1) = \beta, \quad g(1) = 2-\beta\\
    &\mbox{for some $\alpha,\beta\in\R$ with $|\alpha-\beta|\leq \gamma$.}  
  \end{cases}  
  \end{align}   
  
  Indeed, for $(f,g)$ as in~\eqref{eq:discrExPot}, we have  $(f\oplus g)(x,y)=2$ when $x=y$, whereas $(f\oplus g)(0,1)=2+\alpha - \beta$ and $(f\oplus g)(1,0)=2-\alpha + \beta$. In particular, $(f\oplus g)(x,y)\leq 2+\gamma$ when $x\neq y$. As a result, $(f\oplus g-c)_{+}=2\1_{x=y}=:Z\in\cZ$. Now \cref{th:main} shows that $Z$ is the primal optimizer and that $(f,g)$ are potentials. Conversely, let $(f,g)$ be potentials. Define $\alpha:=f(0)$ and $\beta:=f(1)$. Then $(f\oplus g-c)_{+}=2\1_{x=y}$ implies that~$g$ and $\alpha,\beta$ must satisfy the conditions in~\eqref{eq:discrExPot}. In summary, \eqref{eq:discrExPot} describes the family of all potentials when $\gamma\geq0$.
  
  Next, consider $\gamma\in [-2,0)$, or equivalently $c(x,y)= \eta\1_{x\neq y}$ with $\eta\in [0,2)$. Direct calculation shows that $Z_{*}=(1+\eta/2)\1_{x=y} + (1-\eta/2)\1_{x\neq y}$ with constant potentials $(f,g)\equiv(\alpha,1+\eta/2-\alpha)$ for any $\alpha\in\R$. Here the optimal support is the full space, $\spt \pi_{*}=\X\times\Y$, and correspondingly, the identity 
  $Z_{*}=(f\oplus g-c)_{+}=f\oplus g-c$ determines $f\oplus g$ everywhere. 
Let us summarize:
  \begin{enumerate}
  \item For $\gamma>0$, the potentials are non-unique beyond the trivial shift by a constant. A second degree of freedom arises because the two points $(0,0)$ and $(1,1)$ of the support $\spt \pi_{*}$ do not overlap in terms of $\X$ or $\Y$ coordinates. In the language introduced below, the singletons $\{(0,0)\}$ and $\{(1,1)\}$ are the two components of $\spt \pi_{*}$, and they are related to the fact that the potentials are given by a two-parameter family (indexed by $\alpha, \beta$). 
  
  \item For $-2\leq \gamma<0$, we have $\spt \pi_{*}=\X\times\Y$ which has a single component. The potentials span a one-parameter family $(f+\alpha,g-\alpha)_{\alpha\in\R}$; i.e., are as unique as can be.
  
  \item In the boundary case $\gamma=0$, the support still has two components as in~(i), but the two-parameter family degenerates to a one-parameter family since the constraint $|\alpha-\beta|\leq0$ pins down the second parameter to a single value. 
\end{enumerate} 
\end{example}

The following notion of connectedness is the key to generalizing the above observations to arbitrary discrete problems. It was first introduced by~\cite{BorweinLewis.92} in a different context. %

\begin{definition}\label{de:connected}
	Let $E\subset \X\times\Y$ be any subset. Two points $(x,y), (x',y') \in E$ are \emph{connected}, denoted $(x,y) \sim (x',y')$, if there exist $k\in \N_{0}$ and $(x_i,y_{i})_{i=1}^{k}\in E^{k}$ such that the points
\begin{equation}\label{eq:path}
	(x,y), (x_1, y), (x_1, y_1), (x_2, y_1), \dots, (x_{k},y_{k}), (x', y_{k}), (x',y')
\end{equation}
all belong to~$E$. In that case, $(x_i,y_{i})_{i=1}^{k}$ is called a \emph{path} (in~$E$) from $(x,y)$ to $(x',y')$.
The relation~$\sim$ is an equivalence relation on~$E$. The corresponding equivalence classes~$C$ are called the \emph{components} of~$E$, and we denote by $\cC$ the collection of all components.  A set $B\subset E$ is connected (in~$E$) if any two points in~$B$ are connected; thus, the components are the maximal connected subsets of~$E$.
\end{definition}

\begin{figure}[tbh]
    \centering
    \includegraphics[width=.3\textwidth]{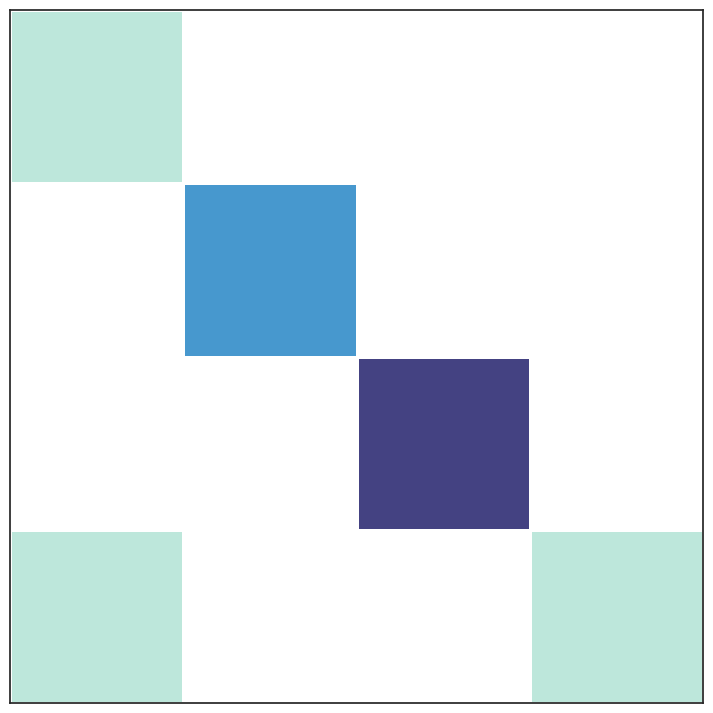}
    \caption{Illustration of a subset $E$ (colored area) of the square $\X\times\Y=[0,1]^{2}$ with three components (color coding).}
    \label{fig:components}
\end{figure}%

For the list~\eqref{eq:path}, the crucial property is that only one coordinate is changed in each step. In our notation, the first coordinate changes first, but because a point can be repeated in the list, this entails no loss of generality.

The present notion of connectedness is graph-theoretic and quite different from the topological one. For instance, two connected subsets are connected to one another as soon as they have points with a common $\X$-coordinate (or $\Y$-coordinate); cf.\ \cref{fig:components}. 

\begin{remark}\label{rk:partition}
  Denote by $\proj_{\X}$ the canonical projection $(x,y)\mapsto x$.
It follows from \cref{de:connected} that for any $(x,y)\in E$, there is exactly one component~$C$ with $x\in\proj_{\X}(C)$. That is, $\{\proj_{\X}(C)\}_{C\in\cC}$ is a partition of~$\proj_{\X}(E)$. The analogue holds for~$\Y$.
\end{remark} 

\begin{lemma}\label{le:conn}
	Let $C\subset E$ be connected. Let $f,f': \proj_{\X}(C)\to\R$ and $g,g': \proj_{\Y}(C)\to\R$ be functions such that $f\oplus g=f'\oplus g'$ on~$C$. Then there exists $\alpha\in\R$ such that $f=f'+\alpha$ on $\proj_{\X}(C)$ and $g=g'-\alpha$ on $\proj_{\Y}(C)$. In particular, the set of all functions $f',g'$ such that $f\oplus g=f'\oplus g'$ on~$C$, is the one-parameter family $(f+\alpha,g-\alpha)_{\alpha\in\R}$.
\end{lemma}

\begin{proof}
  Fix $(x_{0},y_{0})\in C$ and define $\alpha:=f(x_{0})-f'(x_{0})$. 
  Let $(x',y')\in C$. By connectedness, there exists a path $$(x_{0},y_{0}), (x_1, y_{0}), (x_1, y_1), (x_2, y_1), \dots, (x_{k},y_{k}), (x', y_{k}), (x',y')$$ in~$E$. 
  We have $f(x_{0})=f'(x_{0})+\alpha$. As $f\oplus g=f'\oplus g'$ holds at $(x_{0},y_{0})$, it follows that $g(y_{0})=g'(y_{0})-\alpha$. Similarly, $f\oplus g=f'\oplus g'$ holds at $(x_{1},y_{0})$, hence it follows that $f(x_{1})=f'(x_{1})+\alpha$. Continuing inductively, we obtain that $f(x')=f'(x')+\alpha$ and $g(y')=g'(y')-\alpha$.
\end{proof} 

We can now state the main result of this subsection.

\begin{theorem}\label{th:multiplicityDiscrete}
  Let $\mu,\nu$ be finitely supported. Without loss of generality, $\X=\spt \mu$ and $\Y=\spt \nu$. 
  Let $C_{1},\dots,C_{N}$ be the components (cf.\ \cref{de:connected}) of the optimal support $\spt \pi_{*}$ and fix arbitrary potentials~$(f,g)$. The family of all potentials is given by
\begin{align*}
  \left(f + \sum_{i=1}^{N}\alpha_{i}\1_{\proj_{\X}(C_{i})}, ~ g - \sum_{i=1}^{N}\alpha_{i}\1_{\proj_{\Y}(C_{i})}\right), \qquad (\alpha_{1},\cdots,\alpha_{N})\in T
\end{align*} 
where $T\subset\R^{N}$ is the closed convex polytope 
\begin{align*}
  T=\{(\alpha_{1},\cdots,\alpha_{N})\in\R^{N}:\, \alpha_{i}-\alpha_{j} \leq a_{ij}\}
\end{align*} 
containing the origin, with $a_{ij}\in\R_{+}$ given by 
\begin{align}\label{eq:polytopeConstraints}
 a_{ij}:=\inf_{(x,y)\in [\proj_{\X}(C_{i})\times \proj_{\Y}(C_{j})]\setminus \spt \pi_{*}} c(x,y)-f(x)-g(y)\geq0.
\end{align} 
\end{theorem}

\begin{proof}
  By definition,  $(f',g')$ are potentials iff $(f'\oplus g'-c)_{+}=(f\oplus g-c)_{+}$ on $\spt \mu \times \spt\nu$ (which is the whole space $\X\times\Y$ by our assumption). This identity amounts to two requirements:
  
(a)  As $\spt \pi_{*}$ is the subset where the density is strictly positive, we see that $(f'\oplus g'-c)_{+}=(f\oplus g-c)_{+}$ on  $\spt \pi_{*}$ is equivalent to $f'\oplus g'=f\oplus g$ on  $\spt \pi_{*}$. Using \cref{le:conn} and \cref{rk:partition}, the family of all functions $(f',g')$ satisfying $f'\oplus g'=f\oplus g$ on $\spt \pi_{*}$ is
\begin{align*}
  \left(f + \sum_{i=1}^{N}\alpha_{i}\1_{\proj_{\X}(C_{i})}, ~ g - \sum_{i=1}^{N}\alpha_{i}\1_{\proj_{\Y}(C_{i})}\right), \qquad \alpha_{i}\in \R, \; 1\leq i \leq N.
\end{align*} 

(b) The complement of $\spt \pi_{*}$ is the set where the density is zero, hence on that set, $(f'\oplus g'-c)_{+}=(f\oplus g-c)_{+}$ is equivalent to $f'\oplus g'\leq c$. As the sets $\proj_{\X}(C_{i})\times \proj_{\Y}(C_{j})$ form a partition of $\X\times\Y$, $f'\oplus g'\leq c$ with the displayed form for $f'\oplus g'$ is further equivalent to $(f+\alpha_{i})\oplus (g-\alpha_{j})\leq c$ on $[\proj_{\X}(C_{i})\times \proj_{\Y}(C_{j})]\setminus \spt \pi_{*}$ for all $1\leq i,j,\leq N$. This is, in turn, equivalent to $\alpha_{i}-\alpha_{j} \leq a_{ij}$ where $a_{ij}$ is given by~\eqref{eq:polytopeConstraints}. We have $a_{ij}\geq0$ because $(f\oplus g-c)_{+}>0$ on  $\spt \pi_{*}$.
\end{proof}

\Cref{th:multiplicityDiscrete} shows that the potentials form at most an $N$-parameter family, where $N$ is the number of components of the optimal support. We expect in typical cases that they form a nondegenerate $N$-parameter family. However, the boundary case $\gamma=0$ in \cref{ex:discrete} illustrates that the polytope can degenerate to a subset of dimension $<N$.

In the continuous case, on the other hand, it is rather intuitive that the constraint $f'\oplus g'\leq c$ outside $\spt\pi_{*}$ becomes a severe restriction: if the involved functions are continuous, then as $f'\oplus g'= c$ on the boundary of $\spt \pi_{*}$, $|f'\oplus g'- c|$ can be arbitrarily small close to the boundary, leaving no space to vary the~$\alpha_{i}$ (except for the trivial shift). This may serve as an intuition for the situation in \cref{se:contCase} below, where potentials are unique up to the unavoidable additive constant.

\subsubsection{Counterexample for Self-Transport}

In \cref{ex:discrete}, the potentials are unique if we impose the constraint $f=g$. However, the following example shows that potentials need not be unique even if the constraint $f=g$ is imposed in a symmetric setting with $(\X,\mu,\tmu)=(\Y,\nu, \tnu)$ and $c(x,y)=c(y,x)$.

\begin{example}\label{ex:discreteCounter}
  Let $\X=\Y=\{0,1\}$ and $\mu=\tmu=\nu=\tnu=\frac12(\delta_{0}+\delta_{1})$. Let $c(x,y)=(2+\gamma)\1_{x= y}$ where $\gamma>0$. This setting is similar to \cref{ex:discrete}, but the roles of the diagonal and the off-diagonal are interchanged. In analogy to \cref{ex:discrete},
the optimal density is
  $
    Z_{*}(x,y)=2\1_{x\neq y},
  $
  meaning that the optimal coupling~$\pi_{*}$ is the uniform measure on the \emph{off}-diagonal, and $(f,g)$ are potentials if and only if
  \begin{align*}%
  \begin{cases}
    &f(0) = \alpha, \quad g(0) = 2-\beta, \\
    &f(1) = \beta, \quad g(1) = 2-\alpha,\\
    &\mbox{for some $\alpha,\beta\in\R$ with $|\alpha-\beta|\leq \gamma$.}  
  \end{cases}  
  \end{align*}   
  Next, we impose the symmetry constraint $f=g$. This is equivalent to the two equations $\alpha = 2-\beta$ and $\beta = 2-\alpha$, which are however redundant. Then $|\alpha-\beta| \leq \gamma$ is equivalent to $\alpha \in [1-\gamma/2,1+\gamma/2]$ and we conclude that the family of all symmetric potentials $(f,g)$ is
  \begin{align*}%
  \begin{cases}
    &f(0) = g(0) = \alpha, \\
    &f(1) = g(1) = 2-\alpha,\\
    &\mbox{for some $\alpha \in [1-\gamma/2,1+\gamma/2]$.}  
  \end{cases}  
  \end{align*}
  In particular, $f$ is not unique.
\end{example}

\subsection{Continuous Case}\label{se:contCase}

Denote by $\cL^{d}$ the Lebesgue measure on $\R^{d}$. We show that in a regular (yet very standard) setting, the potentials $(f,g)$ are a.s.\ unique up to an additive constant.

\begin{theorem}\label{th:dualUniqueness}
  Let $\X=\Y=\R^{d}$ and $\mu\sim\cL^{d}$ on $\spt\mu$, and assume that 
  $\Int\spt\mu$ is connected.\footnote{This refers to the usual topological connectedness, not \cref{de:connected}.} Moreover, let $c$ be Lipschitz and differentiable on a neighborhood of~$\spt\mu\times\spt\nu$. Let either\footnote{The purpose of this condition is to ensure that any potential is Lipschitz. We could instead assume the latter directly, or assert that uniqueness holds within the class of Lipschitz potentials.}
  \begin{enumerate}
  \item $(\tmu,\tnu)=(\mu,\nu)$ or
  \item $\frac{d\mu}{d\tmu}, (\frac{d\mu}{d\tmu})^{-1}, \frac{d\nu}{d\tnu}, (\frac{d\nu}{d\tnu})^{-1}$ be bounded and Lipschitz, and $c$ be bounded.
  \end{enumerate} 
  Then the potential~$f$ is uniquely determined $\mu$-a.s., up to an additive constant.
\end{theorem}

We recall from \cref{le:gGivenf} that~$g$ is then also uniquely determined $\nu$-a.s., up to constant. In \cref{th:dualUniqueness}, the crucial connectedness assumption is imposed on~$\mu$, whereas~$\nu$ can be quite general. In particular, the result applies in the setting of semi-discrete optimal transport where, typically, $\mu$ is given by a nice ``population'' density and $\nu$ is an empirical measure. The condition on~$\mu$ is satisfied, for instance, if $\mu$ admits a density wrt.~$\cL^{d}$ and $\{d\mu/d\cL^{d}>0\}$ is convex (which implies that $\Int\spt\mu$ is connected).

\begin{proof}[Proof of \cref{th:dualUniqueness}]
  Let $(f,g)$ be potentials. We can extend $c$ to a Lipschitz function on $\X\times\Y$. Recall from \cref{le:unifCont} (or \cref{le:unifContDiffMarg}) that $f,g$ admit versions that are defined everywhere on $\R^{d}$ and Lipschitz. In the remainder of the proof, $f$ and $g$ denote those versions. Note that $Z:=(f\oplus g - c)_{+}$ is a Lipschitz version of the density of the optimal coupling~$\pi_{*}$. In particular, the set $E=\{Z>0\}$ is open and satisfies $\pi_{*}(E)=1$. 
It is then not hard to check that the open set $\proj_{\X}E$ satisfies 
  \begin{align}\label{eq:supportProj}
    \mu(\proj_{\X}E)=1 \qandq \spt \mu = \overline{\proj_{\X}E}.
  \end{align} 

  Let $\dom \nabla f$ denote the set where $f$ is differentiable. By Rademacher's theorem, the complement of $\dom \nabla f$ is $\cL^{d}$-null. 
  On the other hand, it follows from~\eqref{eq:supportProj} and $\mu\sim\cL^{d}$ on $\spt\mu$ that $\cL^{d}(\spt\mu\setminus \proj_{\X}E)=0$. In summary, $\Lambda:=\dom\nabla f\cap\proj_{\X}E\subset\spt\mu$ has full $\cL^{d}$-measure within $\spt\mu$
  
  Next, we check that $\nabla f$ is uniquely determined on~$\Lambda$. Indeed, let $x_{0}\in\Lambda$, then $(x_{0},y_{0})\in E$ for some $y_{0}$. We have
  \begin{align*}
    Z(x,y) = f(x)+ g(y) - c(x,y), \qquad (x,y)\in B_{r}(x_{0},y_{0})
  \end{align*} 
  for small $r>0$, as~$E$ is open. (Here $B_{r}(z)$ denotes the open ball of radius~$r$ around~$z$.) Differentiation thus yields that  $\nabla f(x_{0}) = \nabla_{x} Z(x_{0},y_{0}) + \nabla_{x} c(x_{0},y_{0})$. The right-hand side is uniquely determined. In summary, we have shown that 
  $f$ is a Lipschitz function with $\nabla f$ uniquely determined $\cL^{d}$-a.e.\ on the open and connected set $\Int\spt\mu$. This implies that $f$ is uniquely determined up to additive constant  on $\Int\spt\mu$ (see, e.g., \cite[Formula~2]{Qi.89}). As $\proj_{\X}E$ is open, \eqref{eq:supportProj} implies $\mu(\Int\spt\mu)=1$, completing the proof.
\end{proof} 

\begin{remark}\label{rk:selfCont}
  In \cref{th:dualUniqueness}, suppose that we are in the symmetric setting of self-transport where $(\X,\mu)=(\Y,\nu)$, and we additionally impose that $f=g$. We readily see that this constraint pins down the additive constant in \cref{th:dualUniqueness} and hence gives  uniqueness for~$f$ $\mu$-a.s.
\end{remark} 

\Cref{th:dualUniqueness} is a satisfactory result in a regular setting and covers most applications of interest. In the remainder of the section, we comment briefly on subtleties that occur in a possible extension of the analysis that we performed in the discrete case to non-discrete (and irregular) cases. First, while the definition of components applies to arbitrary sets, two almost-surely equal sets may have substantially different components. Second, some components may carry no mass, making them negligible. Next, we highlight the first aspect by detailing a continuous version of \cref{ex:discrete} and comparing with another version.

\begin{example}\label{ex:cont}
  Let $\X=\Y=[0,1]$ and let $\mu=\tmu=\nu=\tnu=\cL^{1}|_{[0,1]}$ be the uniform measure. Let $E:=[0,1/2)^{2}\cup (1/2,1]^{2}$ be the ``block diagonal'' and $E^{c}=[0,1]^{2}\setminus E$.
  Let $c(x,y)=(2+\gamma)\1_{E^{c}}$ where $\gamma>0$. We claim that
  \begin{align*}
    Z_{*}(x,y)=2\1_{E}
  \end{align*} 
  is a version of the optimal density, meaning that the optimal coupling~$\pi_{*}$ is the uniform measure on~$E$, and that $(f,g)$ are potentials if and only if $\cL^{1}$-a.s.,
  \begin{align}\label{contExPot}
  \begin{cases}
    &f(x) = \alpha, \quad g(y) = 2-\alpha, \qquad x,y\in [0,1/2), \\
    &f(x) = \beta, \quad g(y) = 2-\beta, \qquad x,y\in (1/2,1],\\
    &\mbox{for some $\alpha,\beta\in\R$ with $|\alpha-\beta|\leq \gamma$.}  
  \end{cases}  
  \end{align} 
  
  Indeed, for such $(f,g)$, we have  $f\oplus g=2$ on $E$, whereas $(f\oplus g)\in \{2+\alpha - \beta, 2-\alpha + \beta\}$ on $E^{c}$. In particular, $f\oplus g\leq 2+\gamma$ on $E^{c}$. As a result, $(f\oplus g-c)_{+}=2\1_{E}=Z_{*}$. Now \cref{th:main} shows that $Z_{*}$ is the primal optimizer and that $(f,g)$ are potentials. Conversely, let $(f,g)$ be potentials. As $f(x)+g(y)=2$ for $\cL^{2}$-a.e.\ $(x,y)\in [0,1/2)^{2}$, it follows that $f(x)=2-2\int_{0}^{1/2}g(y)\,dy=:\alpha$ for $\mu$-a.e.\ $x\in [0,1/2)$ and similarly $g(y)=2-2\int_{0}^{1/2}f(x)\,dx = 2-\alpha$ for $\nu$-a.e.\ $y\in [0,1/2)$. Analogously, $f=\beta$ $\mu$-a.s.\ on $(1/2,1]$ and $g=2-\beta$ $\nu$-a.s.\ on $(1/2,1]$, for some $\beta\in\R$. We then have $f\oplus g=2+\alpha - \beta$ $P$-a.s.\ on $[0,1/2)\times(1/2,1]$ and $f\oplus g=2-\alpha+\beta$ $P$-a.s.\ on $(1/2,1]\times[0,1/2)$. In order to satisfy $(f\oplus g-c)_{+}=Z_{*}=2\1_{E}$ $P$-a.s., we must have $|\alpha-\beta|\leq \gamma$. 
\end{example} 

In \cref{ex:cont}, the number of components of $\{Z_{*}>0\}$ correctly describes the degrees of freedom in choosing potentials. To achieve this, we picked a ``good'' version of $Z_{*}$---the following shows that the number of components of $\{Z_{*}>0\}$ can depend on the version of the density $Z_{*}$. While the ``support'' of $\pi_{*}$ has a canonical definition in the discrete case, that is not the case here (and the topological support is not a good choice). 

\begin{remark}\label{rk:contExample}
  In \cref{ex:cont}, the family in~\eqref{contExPot} are precisely the functions $f,g$ such that 
  $(f\oplus g-c)_{+}=2\1_{E}$ \emph{everywhere} on $[0,1]^{2}$ (without exceptional nullsets). 
  Next, consider the same example but define $\tilde{E}=[0,1/2]^{2}\cup [1/2,1]^{2}$, which is the topological support of~$\pi_{*}$. Obviously $\tilde E = E$ $P$-a.s., hence $Z_{*}(x,y)=2\1_{\tilde{E}}$ is another version of the optimal density. But by contrast with the above, $(f\oplus g-c)_{+}=2\1_{\tilde E}$ \emph{everywhere} on $[0,1]^{2}$ if and only if $f\equiv \alpha$ and $g\equiv 2-\alpha$ for some $\alpha\in\R$. To wit, the extra degree of freedom represented by $\beta$ has been eliminated, as $f(1/2)=\alpha$ and $f(1/2)=\beta$ now imply $\alpha=\beta$ (or similarly for $g$). The same happens if we replace $\tilde E$ by $[0,1/2]\times[0,1/2) \cup [1/2,1]\times(1/2,1]$ or the symmetric counterpart. Indeed, $\tilde E$ and the latter two sets have a single component.
  
  Due to the closure in its definition, the topological support is a rather large set carrying the measure, whereas for the current purpose we require a small set, or more precisely, a set with the \emph{maximum number of non-negligible components.}
\end{remark} 

\section{Sparsity for Small Regularization}\label{se:sparseQuadratic}

Let $\mu,\nu$ be compactly supported probability measures on $\R^{d}$. We specialize our general setting to $\X=\spt\mu$ and $\Y=\spt\nu$, with cost $c(x,y)=\|x-y\|^{2}$ the squared Euclidean distance. We further assume that $\mu\ll\cL^{d}$ and that $\Int \spt \mu$ is connected, and choose $(\tmu,\tnu)=(\mu,\nu)$ for simplicity (alternately, we can impose the conditions of \cref{le:unifContDiffMarg}).

In this setting of quadratic cost, it is well known  that the unregularized transport problem
\begin{align}\label{eq:quadUnregPrimal}
   \sP_{0} = \inf_{\pi\in\Pi(\mu,\nu)} \int c\,d\pi
\end{align}
has a unique optimal coupling $\pi_{0}$, given by Brenier's map. See, e.g., \cite{Villani.09} for background. In particular, $\pi_{0}$ is concentrated on the graph of a function $T: \R^{d}\to\R^{d}$, and hence ``sparse'' (as sparse as a coupling can be). Regularity results for~$T$ are known under conditions on the marginals.

Let $\pi_{\eps}\in\Pi(\mu,\nu)$ be the optimal coupling of the quadratically regularized problem with regularization parameter $\eps>0$ (cf.\ \cref{rk:regularizationParam}),
\begin{align*}
  \sP_{\eps} = \inf_{\pi\in\Pi(\mu,\nu)} \int c\,d\pi + \frac{\eps}{2} \|d\pi/dP\|^{2}_{L^{2}(P)}.
\end{align*}
The next result formalizes and establishes that $\pi_{\eps}$ is sparse for small $\eps$, by showing that $\spt \pi_{\eps}$ is contained in a small neighborhood of the sparse set $\spt \pi_{0}$. To the best of our knowledge, it is the first theoretical result showing sparsity of quadratically regularized  transport in a continuous setting.

\begin{theorem}[Sparsity for quadratic cost]\label{th:sparse}
  Let $U$ be an open neighborhood of $\spt \pi_{0}$. Then $\spt \pi_{\eps}\subset U$ for all sufficiently small $\eps>0$.
\end{theorem} 

\begin{proof}
  We recall that in the present setting, the unregularized transport problem~\eqref{eq:quadUnregPrimal} admits a unique optimal coupling $\pi_{0}$ and a unique (up to constant) Kantorovich potential $f:\X\to\R$ that is Lipschitz continuous (see, e.g., \cite[Theorem~5.10]{Villani.09} for existence and \cite[Appendix~B]{BerntonGhosalNutz.21} for uniqueness). Fix $x_{0}\in\X$; then we may normalize $f(x_{0})=0$ to have uniqueness. Let $g$ be the $c$-concave conjugate of $f$, so that $(f,g)$ is the solution of the dual optimal transport problem
\begin{align}\label{eq:dualUnregOT}
  \sD_{0}=\sup_{(f,g)\in C(\X)\times C(\Y):\,  f\oplus g \leq c} \;\int f \,d\mu + \int g \,d\nu.
\end{align}   
It is known that $\spt \pi_{0}=\{f\oplus g = c\}$. Indeed, $\pi_{0}\{f\oplus g = c\}=1$ holds for general costs. The inclusion $\spt \pi_{0}\supset\{f\oplus g = c\}$, which is crucial below, follows because the section $\{f(x)+ g(\cdot) = c(x,\cdot)\}\subset\Y$ is a singleton for $\mu$-a.e.\ $x\in\X$ (as the subdifferential of an a.e.\ differentiable function).

Let $(f_{\eps},g_{\eps})$ be the rescaled potentials as defined in \cref{rk:regularizationParam}. By \cref{le:unifCont}\,(iii), $f_{\eps}$ can be chosen to be $L$-Lipschitz, where $L$ is the Lipschitz constant of $c$ on the compact set $\X\times\Y$ (note that $f_{\eps}/\eps$ is the potential for $c/\eps$ without rescaling). We may normalize $f_{\eps}(x_{0})=0$ and then $f_{\eps}$ is also bounded uniformly in $\eps$. The Arzela--Ascoli theorem shows that given a sequence $\eps_{n}\to0$, a subsequence of $(f_{\eps_{n}})$ converges uniformly to some limit~$f_{*}$. After passing to another subsequence (still denoted $\eps_{n}$), we also have uniform convergence $g_{\eps_{n}}\to g_{*}$. We show below that $(f_{*},g_{*})=(f,g)$. In particular, the uniqueness of the Kantorovich potential then implies that $(f_{\eps},g_{\eps})\to(f,g)$ uniformly for $\eps\to0$.
  
  Let $U$ be an open neighborhood of $\spt \pi_{0}$. As $\spt \pi_{0}=\{f\oplus g = c\}$ and $f,g,c$ are continuous and $\X\times\Y$ is compact, there exists $\delta>0$ such that $\{f\oplus g \geq c-\delta\}\subset U$. Recall that the density of $\pi_{\eps}$ has the form
\begin{align}\label{eq:ZepsForm}
  Z_{\eps} = \left(\frac{f_{\eps}\oplus g_{\eps} - c}{\eps}\right)_{+}.
\end{align}
In view of the uniform convergence $(f_{\eps},g_{\eps})\to(f,g)$, there exists $\eps_{0}>0$ such that for all $\eps\in(0,\eps_{0})$, 
  \begin{align*}
  \spt \pi_{\eps} = \overline{\{Z_{\eps}>0\}} = \overline{ \{f_{\eps}\oplus g_{\eps} > c\}} \subset \{f\oplus g \geq c-\delta\}\subset U.
\end{align*}

It remains to prove $(f_{*},g_{*})=(f,g)$. To that end, we show that $(f_{*},g_{*})$ solves the dual problem~\eqref{eq:dualUnregOT}.  
We first verify that $(f_{*},g_{*})$ is in the dual domain; i.e., that $f_{*}\oplus g_{*}\leq c$. Suppose that $f_{*}(x)+g_{*}(y)>c(x,y)$ at some $(x,y)\in \X\times\Y$. Then by continuity of $f_{*}$, $g_{*}$ and $c$, the set $\{f_{*}\oplus g_{*}>c\}$ includes a compact neighborhood~$B$ of~$(x,y)$. In view of~\eqref{eq:ZepsForm}, it follows that $Z_{\eps} \to\infty$ uniformly on~$B$ as $\eps\to0$. As $(\mu\otimes\nu)(B)>0$ due to $(x,y)\in \X\times\Y=\spt (\mu\otimes\nu)$, this contradicts the fact that $ \int Z_{\eps} \,d(\mu\otimes\nu)=1$ for all $\eps>0$.

Second, we verify that $(f_{*},g_{*})$ is optimal for~\eqref{eq:dualUnregOT}. By duality we have $\sD_{\eps}=\sP_{\eps}$ and $\sP_{0}=\sD_{0}$, and clearly $\sP_{0}\leq\sP_{\eps}$ as the quadratic penalty is nonnegative. Thus also $\sD_{0}\leq\sD_{\eps}$, which by~\eqref{eq:dualEpsRescaled} yields
\begin{align*}%
  \sD_{0}\leq \sD_{\eps} 
  &= \int f_{\eps}\,d\mu + \int g_{\eps}\,d\nu - \frac{\eps}{2}\int \left(\frac{f_{\eps}\oplus g_{\eps} - c}{\eps}\right)_{+}^{2}\,dP \\
  &\leq \int f_{\eps}\,d\mu + \int g_{\eps}\,d\nu \to \int f_{*}\,d\mu + \int g_{*}\,d\nu.
\end{align*}
Hence $(f_{*},g_{*})$ solves~\eqref{eq:dualUnregOT} and the proof is complete.
\end{proof} 

\begin{remark}\label{rk:onSparcityResult}
  (a) \Cref{th:sparse} has a trivial analogue in the discrete setting, where it is known that $\pi_{\eps}=\pi_{0}$ for small $\eps>0$, for a certain optimal transport~$\pi_{0}$. See, e.g., \cite{DesseinPapadakisRouas.18}. This observation goes back to \cite{MangasarianMeyer.79} for more general linear programs with quadratic regularization. A quantitative version, characterizing the maximal~$\eps$ with $\pi_{\eps}=\pi_{0}$, was recently given in~\cite{GonzalezSanzNutz.24}. Still in the discrete setting, \cite{GonzalezSanzNutzRiveros.24} shows that the support of $\pi_{\eps}$ need not shrink monotonically in $\eps$. 
  
  (b) As $\pi_{\eps}\to\mu\otimes\nu$ for $\eps\to\infty$, sparsity certainly requires $\eps$ to be small in some sense.  One can hope for a quantitative version of \cref{th:sparse}, stating that $\spt \pi_{\eps}$ is in a $\delta$-neighborhood of $\spt \pi_{0}$, where $\delta=\delta(\eps)$ has an explicit dependence on $\eps$. This problem is left for future research. The proof given above merely uses the straightforward (qualitative) convergence of the potentials; see also \cite{GigliTamanini.21, NutzWiesel.21} for related results on the convergence of potentials for entropic regularization. The proof does extend to more general costs~$c$: the key property is that the Kantorovich dual $f\oplus g$ ``detaches'' from the cost $c$ outside the support of the optimal coupling $\pi_{0}$. See \cite{CarlierPegonTamanini.22, MalamutSylvestre.23} for recent developments in this direction. 
\end{remark} 

\section{Proof of \cref{th:main}\hspace{.05em}{(i)--(iii)}}\label{se:proofMain}

We first take care of generalities---primal existence, automatic integrability and weak duality---which will reduce the proof of \cref{th:main} to the main task, namely to show that the optimal density is of the form $Z_{*}=(f\oplus g - c)_{+}$.

\subsection{Primal Existence}\label{se:primalEx}

The existence and uniqueness of the primal optimizer is straightforward and well known. We detail the argument for later use.

\begin{proof}[Proof of \cref{th:main}(ii)]
Recall~\eqref{eq:cIntCond}. For any $Z\in\cZ$, clearly 
\begin{align*}
  \int cZ\,dP + \frac12 \|Z\|^{2}_{L^{2}(P)} 
  &= \int (cZ+ \tfrac12 Z^{2})\,dP \\
  &= \tfrac12 \|c+Z\|_{L^{2}(P)}^{2} -\tfrac12 \|c\|_{L^{2}(P)}^{2},
\end{align*} 
where the last term is a finite constant independent of $Z$.
In particular,
\begin{align}\label{eq:primalProj}
    \sP = \inf_{Z\in\cZ} \tfrac12 \|c+Z\|_{L^{2}(P)}^{2} -\tfrac12 \|c\|_{L^{2}(P)}^{2}.
\end{align}
The subset $\cZ\subset L^{2}(P)$ is closed, convex, and nonempty as  $d(\mu\otimes\nu)/dP\in\cZ$ due to~\eqref{eq:refMeasureCond}. The claim thus follows from the existence, uniqueness and characterization of the Hilbert space projection onto~$\cZ$.
\end{proof} 

\subsection{Integrability Properties}\label{se:integrabilityProperties}

This subsection establishes the automatic integrability of potentials. 
We first derive a simple lower bound similar to a result in~\cite{LorenzMannsMeyer.21}. See~\eqref{eq:cIntCond} for the notation~$c_{1},c_{2}$. %

\begin{lemma}[Lower bound]\label{le:lowerBounds}
    Let $(f\oplus g -c)_{+}$ be the $P$-density of a coupling $\pi\in\Pi(\mu,\nu)$, for some measurable functions $f:\X\to\R$ and $g:\Y\to\R$. Then
  \begin{align*}
    (f-c_{1})_{-}\in L^{\infty}(\tmu), \qquad (g-c_{2})_{-}\in L^{\infty}(\tnu). 
  \end{align*}  
  In particular, $f_{-}\in L^{1}(\tmu)\cap L^{1}(\mu)$ and $g_{-}\in L^{1}(\tnu)\cap L^{1}(\nu)$.
  If $c_{-}\in L^{\infty}(P)$, then $(f_{-},g_{-})\in L^{\infty}(\tmu)\times L^{\infty}(\tnu)$. 
  
  Moreover, if $c \in L^{\infty}(P)$, then $(f-\frac{d\mu}{d\tmu})_{+}\in L^{\infty}(\tmu)$ and $ (g-\frac{d\mu}{d\tmu})_{+} \in L^{\infty}(\tnu)$.
\end{lemma}

\begin{proof}
  Recall~\eqref{eq:conjugateDefn} and~\eqref{eq:cIntCond}. Our first goal is to prove that 
  \begin{align}\label{eq:lowerBoundPrf}
    (f(x)-c_{1}(x))_{-} \leq \Phi^{-1}\left(\frac{d\mu}{d\tmu}(x)\right) \qquad\mbox{for $\tmu$-a.e.\ $x\in\X$,}
  \end{align} 
  for a nonnegative function $\Phi^{-1}$ to be defined below. This holds trivially for $x\in\X$ such that $f(x)-c_{1}(x)\geq 0$, as that makes the left-hand side vanish. Whereas for $\tmu$-a.e.\ $x\in\X$ such that $t:=f(x)$ satisfies $t-c_{1}(x)\leq 0$, we have 
  \begin{align*}
     F_{x}(t)
     &=\int [t +g(y)-c(x,y)]_{+}\,\tnu(dy) \\
     &\leq \int (t-c_{1}(x) + g(y)-c_{2}(y))_{+}\,\tnu(dy)\\
     &= \int [ g(y)-c_{2}(y)  - (t-c_{1}(x))_{-}]_{+}\,\tnu(dy)\\
     &\leq \int [(g(y)-c_{2}(y))_{+} - (t-c_{1}(x))_{-}]_{+}\,\tnu(dy)\\
     &\leq \int h(y) \1_{h(y) > (t-c_{1}(x))_{-}}\,\tnu(dy)  = \Phi((t-c_{1}(x))_{-})
  \end{align*}  
  where $h(y):=(g(y)-c_{2}(y))_{+}$ and $\Phi(\alpha) := \int h(y)\1_{h(y)>\alpha}\,\tnu(dy)$.   As $c_{2}\in L^{1}(\tnu)$ and, necessarily, $g_{+}\in L^{1}(\tnu)$, we have $h\in L^{1}(\tnu)$. In particular, $\Phi$ is a nonincreasing function with $\lim_{\alpha\to\infty} \Phi(\alpha)=0$. Define the generalized inverse $\Phi^{-1}(v):=\sup\{u:\, \Phi(u)\geq v\}$, which is a nonincreasing function  $(0,A)\to \R_{+}$ for $A:=\int h(y)\1_{h(y)>\alpha}\,\tnu(dy)$. It follows with $\frac{d\mu}{d\tmu}(x)=F_{x}(t)$ for $t=f(x)$ that~\eqref{eq:lowerBoundPrf} holds. (The above chain of inequalities ensures that $\frac{d\mu}{d\tmu}(x)=F_{x}(t)$ is in the domain of $\Phi^{-1}$.)

  As we have assumed in~\eqref{eq:refMeasureCond} that $\frac{d\mu}{d\tmu}$ is a.s.\ uniformly bounded away from zero, \eqref{eq:lowerBoundPrf} and the monotonicity of $\Phi^{-1}$ show that $(f-c_{1})_{-}\in L^{\infty}(\tmu)$. Recalling that $c_{1}\in L^{1}(\mu)\cap L^{1}(\tmu)$, this also implies $f_{-}\in L^{1}(\tmu)\cap L^{1}(\mu)$.
  
  The claims that $(g-c_{2})_{-}\in L^{\infty}(\tnu)$ and $g_{-}\in L^{1}(\tnu)\cap L^{1}(\nu)$ now follow by symmetry.
  
  Suppose that $c \in L^{\infty}(P)$. Then 
  \[
    \frac{d\mu}{d\tmu}(x)=F_{x}(t)=\int (t +g(y)-c(x,y))_{+}\,\tnu(dy) \geq t - \|g_{-}\|_{L^{\infty}(\tnu)} - \|c\|_{L^{\infty}(P)}
  \]
  for $t=f(x)$ implies $f(x)-\frac{d\mu}{d\tmu}(x) \leq \|g_{-}\|_{L^{\infty}(\tnu)} + \|c\|_{L^{\infty}(P)}$, and it was shown above that $\|g_{-}\|_{L^{\infty}(\tnu)}$ is finite when $c \in L^{\infty}(P)$.
\end{proof}

The next lemma will entail in particular that the algebraic form $(f\oplus g -c)_{+}$ already identifies the optimal coupling, without any (a priori) integrability condition on $(f,g)$. As a consequence, the system of equations in \cref{th:main}\,(d) fully characterizes potentials, without a separate condition.

\begin{lemma}[Automatic integrability]\label{le:fgIntegrable}
  Let $Z=(f\oplus g -c)_{+}$ be the $P$-density of a coupling $\pi\in\Pi(\mu,\nu)$, for some measurable functions $f:\X\to\R$ and $g:\Y\to\R$. Then %
  \begin{enumerate}
  \item $(f_{+},g_{+})\in L^{2}(\tmu)\times L^{2}(\tnu)$,
  \item $Z\in\cZ$; i.e., $Z\in L^{2}(P)$,
  \item $(f,g)\in L^{1}(\mu)\times L^{1}(\nu)$.
  \end{enumerate}
\end{lemma}

\begin{proof}%
Recall that the equivalence of (b) and (d) in \cref{th:main}\,(iii) was already proved (after \cref{rk:regularizationParam}). Hence we have from~\eqref{eq:marginalEqmu} that for $\tmu$-a.e.\ $x\in\X$,
\begin{align}
  \frac{d\mu}{d\tmu}(x) 
  & = \int (f(x)+g(y)-c(x,y))_{+}\,\tnu(dy) \nonumber\\
  & \geq f(x) - \int g_{-} \,d\tnu  - \int c(x,y)\,\tnu(dy) \label{eq:fUpperL2bound}.
\end{align} 
We have $\int g_{-} \,d\tnu<\infty$ by \cref{le:lowerBounds}, and the last term is in $L^{2}(\tmu)$ by Jensen:
\[
\int \left(\int c(x,y)\,\tnu(dy)\right)^{2}\tmu(dx) \leq \iint c(x,y)^{2}\,\tnu(dy)\tmu(dx) = \|c\|^{2}_{L^{2}(P)}<\infty.
\]
Hence~\eqref{eq:fUpperL2bound} and~\eqref{eq:refMeasureCond} establish that~$f$ is bounded from above by a function in $L^{2}(\tmu)$; that is, $f_{+}\in L^{2}(\tmu)$. Analogously, $g_{+}\in L^{2}(\tnu)$, completing~(i).

As $c\in L^{2}(P)$, it follows that $Z=(f(x)+g(y)-c(x,y))_{+}\in L^{2}(P)$, which is (ii). The Cauchy--Schwarz inequality and~\eqref{eq:refMeasureCond} yield $L^{2}(\tmu) \subset L^{1}(\mu)$ and $L^{2}(\tnu) \subset L^{1}(\nu)$. Hence~(i) implies $(f_{+},g_{+})\in L^{1}(\mu)\times L^{1}(\nu)$. On the other hand, $(f_{-},g_{-})\in L^{1}(\mu)\times L^{1}(\nu)$ by \cref{le:lowerBounds}, showing (iii).
\end{proof}

\subsection{Weak Duality}

The next lemma provides weak duality and reduces strong duality to existence of potentials.

\begin{lemma}\label{le:weakDuality}
  \begin{enumerate}
  \item
  For all $Z\in\cZ$ and $(f,g)\in L^{1}(\mu)\times L^{1}(\nu)$,
  \begin{align*}
    \int \left(cZ + \tfrac12 Z^{2}\right)  \,dP \geq \int f\,d\mu + \int g\,d\nu  - \frac12 \int (f\oplus g - c)_{+}^{2} \, dP,
  \end{align*} 
  with equality holding iff $Z=(f\oplus g - c)_{+}$ $P$-a.s.  
  In particular, $\sP\geq \sD$.
  
  \item
  Let $f:\X\to\R$ and $g:\Y\to\R$ be measurable and suppose that $Z:=(f\oplus g - c)_{+}$ is the density of a coupling. Then
  \begin{enumerate}
  \item $Z\in\cZ$ and $(f,g)\in L^{1}(\mu)\times L^{1}(\nu)$,
    \item there is no duality gap: $\sP=\sD$,
  \item $Z$ is optimal for the primal problem~\eqref{eq:primal},
  \item $(f,g)$ is optimal for the dual problem~\eqref{eq:dual}.
  \end{enumerate} 
  
  \item 
  Conversely, suppose that $\sP=\sD$. If $(f,g)\in L^{1}(\mu)\times L^{1}(\nu)$ is optimal for the dual problem~\eqref{eq:dual}, then $(f\oplus g - c)_{+}$ is in~$\cZ$ and optimal for the primal problem~\eqref{eq:primal}.
  \end{enumerate}
\end{lemma} 

\begin{proof}
Consider $Z\in\cZ$ and $(f,g)\in L^{1}(\mu)\times L^{1}(\nu)$. Then
\begin{align*}
  \int \left(cZ + \tfrac12 Z^{2}\right)  \,dP
  = \int f\oplus g \,d(\mu\otimes\nu) - \int\left((f\oplus g-c)Z -\tfrac12 Z^{2}\right)  \,dP
\end{align*}
as $Z\,dP\in \Pi(\mu,\nu)$.
Note that $[0,\infty)\ni z\mapsto az-z^{2}/2$ has a unique maximum at $z=a_{+}$ with maximum value $a_{+}^{2}/2$. Using this pointwise with $a=f\oplus g-c$, we deduce
\begin{align*}
  \int \left(cZ + \tfrac12 Z^{2}\right)  \,dP
  \geq \int f\oplus g \,d(\mu\otimes\nu) - \frac12 \int (f\oplus g - c)_{+}^{2} \, dP
\end{align*} 
with equality holding if and only if  $Z=(f\oplus g-c)_{+}$ $P$-a.s. This shows~(i). In view of the automatic integrability shown in \cref{le:fgIntegrable}, (ii) follows from~(i).  To see (iii), consider $Z:=Z_{*}\in\cZ$ in the left-hand side of in~(i) and dual optimizers $(f,g)$ on the right-hand side. As $\sP=\sD$ was assumed, the assertion of~(i) on equality implies $Z_{*}=(f\oplus g-c)_{+}$ $P$-a.s.
\end{proof} 

\subsection{Construction of Potentials}

In view of \cref{le:weakDuality}, our main task is to construct measurable functions $f:\X\to\R$ and $g:\Y\to\R$ such that $(f\oplus g - c)_{+}\in\cZ$. More precisely, we shall show directly that the optimal density $Z_{*}$ is of that form. 

From a convex programming point of view, the marginal constraints $\mu,\nu$ in the primal problem~\eqref{eq:primal} correspond to infinitely many equality constraints; namely, $\int \phi \,d\pi=\int \phi \,d\mu$ whenever $\phi$ is a bounded measurable function on $\X$, and similarly for~$\nu$. As the spaces $\X,\Y$ are separable, countably many test functions~$\phi$ are sufficient to encode the marginals. Our plan is to approximate the primal problem~\eqref{eq:primal} with auxiliary problems having finitely many constraints (that can be solved by elementary arguments) and then pass to the limit (which is more delicate).

The problems with finitely many constraints are described in the next lemma. We emphasize that $\cZ_{n}$ consists of densities of measures that are not necessarily probability measures. 

\begin{lemma}\label{le:finiteDimProblem}
  Fix $n\in\N$ and bounded measurable functions $\phi_{i}: \X\times\Y\to \R$, $1\leq i\leq n$ with $\int \phi_{i}\,d(\mu\otimes\nu)=0$.
   Let
  $$
    \cZ_{n}=\left\{Z\in L^{2}(P): Z\geq0, \, \int Z \,dP=1, \, \int \phi_{i}Z \,dP=0,\,  1\leq i\leq n\right\}.
  $$
  There is a unique solution $Z_{n}\in \cZ_{n}$ to
  \begin{align}\label{eq:primalfinitelyManyConstraints}
  \inf_{Z\in\cZ_{n}} \int cZ\,dP + \frac12 \|Z\|^{2}_{L^{2}(P)}
  \end{align}
  and $Z_{n}$ is characterized within $\cZ_{n}$ by being of the form
  \begin{equation}\label{eq:finitelyManyConstraintsForm}
    Z_{n}= (b_{0} + b_{1}\phi_{1}+\dots +b_{n}\phi_{n} -c)_{+}\qquad \mbox{for some $b_{i}\in\R$, } 0\leq i\leq n.
  \end{equation}
\end{lemma}

\begin{proof}
  As $\cZ_{n}\subset L^{2}(P)$ is convex, closed and nonempty, existence and uniqueness of the minimizer $Z_{n}=\argmin_{Z\in\cZ_{n}} \|Z+c\|_{L^{2}(P)}$ follow by Hilbert space projection as in the proof of \cref{th:main}\,(ii). Next, we argue as in the proof of \cref{le:weakDuality}\,(i), with $f\oplus g$ replaced by $b_{0}+b\cdot \Phi$ where $b_{0}\in\R$ and  $b=(b_{1},\dots,b_{n})\in\R^{n}$ and $\Phi=(\phi_{1},\dots,\phi_{n})$. Noting that $\int(b_{0}+b\cdot \Phi) \,d(\mu\otimes\nu)=b_{0}$, we obtain
\begin{align*}
  \int \left(cZ + \tfrac12 Z^{2}\right)  \,dP
  \geq b_{0} - \frac12 \int (b_{0}+ b\cdot \Phi-c)_{+}^{2} \, dP,
\end{align*} 
with equality holding if and only if $Z=(b_{0}+b\cdot \Phi-c)_{+}$ $P$-a.s.\ for some $b_{0}\in\R$ and $b\in\R^{n}$. 
As a result, we only need to prove that there exists $Z\in\cZ_{n}$ of the form $Z=(b_{0}+b\cdot \Phi-c)_{+}$.
To that end, we first show that the problem
\begin{equation*}
  \inf_{(b_{0},b)\in\R^{n+1}} G(b_{0},b), \qquad G(b_{0},b):= -b_{0} + \frac12\int (b_{0} + b\cdot\Phi - c)_{+}^{2}\,dP
\end{equation*}
admits a minimizer~$\boldsymbol{b}^{*}=(b_{0}^{*},b^{*})$. Note that $G:\R^{n+1}\to\R$ is convex and continuous. 
By projecting onto the orthogonal complement of $\{b\in\R^{n}: b\cdot\Phi=0~P\as\}$, we can reduce to a situation where $b\cdot\Phi=0$ $P$-a.s.\ only for $b=0$. For $b\neq0$ we then have $P\{b\cdot\Phi\neq0\}>0$.

We claim that this implies $P\{b\cdot\Phi>0\}>0$. Indeed, if not, then $b\cdot\Phi\leq0$ $P$-a.s.\ and $P\{b\cdot\Phi<0\}>0$. As 
$\mu\otimes\nu\sim P$, 
it follows that $b\cdot\Phi\leq0$ $(\mu\otimes\nu)$-a.s.\ and
$(\mu\otimes\nu)\{b\cdot\Phi<0\}>0$. Thus $\int b\cdot\Phi \, d(\mu\otimes\nu)<0$, contradicting that $\int b\cdot\Phi \, d(\mu\otimes\nu)=0$. 

Using that $P\{b\cdot\Phi>0\}>0$ for any $b\neq0$, we verify the radial coercivity condition 
$$
  \lim_{\alpha\to\infty} G(\alpha \boldsymbol{b})=\infty, \qquad  0\neq \boldsymbol{b}\in\R^{n+1}. 
$$
Any convex, lower semicontinuous function $G:\R^{n+1}\to\R$ satisfying this condition attains its minimum \cite[Lemma~3.5, p.\,126]{FollmerSchied.11}.
  
   Let~$\boldsymbol{b}^{*}=(b_{0}^{*},b^{*})$ be a minimizer. Note that $z\mapsto z_{+}^{2}$ is continuously differentiable with derivative $2z_{+}$, and recall that $\Phi$ is bounded. Differentiation with respect to $b_{i}$ under the integral yields the first-order conditions 
  $$
   \int (b^{*}_{0}+b^{*}\cdot\Phi-c)_{+}\,dP=1, \qquad \int \phi_{i} (b^{*}_{0}+b^{*}\cdot\Phi-c)_{+}\,dP=0, \quad 1\leq i \leq n.
  $$
  This shows that $(b^{*}_{0}+b^{*}\cdot\Phi-c)_{+} \in \cZ_{n}$, as desired.
\end{proof}

\begin{lemma}\label{le:approxSeq}
  Let $Z_{*}\in\cZ$ be the primal optimizer. There exist bounded measurable functions $f_{n}:\X\to\R$ and $g_{n}:\Y\to\R$, $n\geq1$ such that
  \begin{align*}
    Z_{*} = \lim_{n\to\infty} (f_{n}\oplus g_{n} - c)_{+}\qquad\mbox{$P$-a.s.\ and in $L^{2}(P)$.}
  \end{align*} 
\end{lemma}

\begin{proof}
As $\X$ is Polish, there are (countably many) bounded measurable functions $\phi^{\mu}_{k}: \X\to\R$, $k\geq1$ such that $\rho\in\cP(\X\times\Y)$ satisfies $\int \phi^{\mu}_{k}(x)\,\rho(dx,dy)=0$ for all $k\geq1$ if and only if the first marginal of $\rho$ is~$\mu$ (cf.\ \cref{rk:minimalAssumptions}). 
Similarly, there are functions $\phi^{\nu}_{k}: \Y\to\R$ for $\nu$. Let $\phi_{2i-1}(x,y):=\phi^{\mu}_{i}(x)$ and $\phi_{2i}(x,y):=\phi^{\nu}_{i}(y)$, so that $\rho\in\Pi(\mu,\nu)$ if and only if $\int \phi_{i}\,d\rho=0$ for all $i\geq1$. Define $Z_{n},\cZ_{n}$ as in \cref{le:finiteDimProblem} and note that $Z_{n}$ is of the desired form $Z_{n}=(f_{n}\oplus g_{n} -c)_{+}$; namely, $f_{n}$ is a linear combination of $(\phi^{\mu}_{k})_{k\leq n}$ and $g_{n}$ is a linear combination of $(\phi^{\nu}_{k})_{k\leq n}$. Thus, it suffices to show that $Z_{n}\to Z_{*}$ in $L^{2}(P)$.

Note that $\cZ_{n}, \cZ \subset L^{2}(P)$ are closed and convex, $\cZ_{n}\supset \cZ_{n+1}$, and $\cZ=\cap_{n\geq1} \cZ_{n}$. Moreover, $Z_{n}$ and $Z_{*}$ are the projections of $-c$ onto $\cZ_{n}$ and $\cZ$ in $L^{2}(P)$, respectively:
\begin{align*}
  Z_{n}=\argmin_{Z\in\cZ_{n}} \|Z+c\|, \qquad Z_{*}=\argmin_{Z\in\cZ} \|Z+c\|,
\end{align*} 
where $\|\cdot\|=\|\cdot\|_{L^{2}(P)}$.
It is a general fact of Hilbert spaces that such nested projections converge; i.e., $Z_{n}\to Z_{*}$ in $L^{2}(P)$.
One way of obtaining that fact is to use the parallelogram law for $Z_{m}+c$ and $Z_{n}+c$, giving
\[
  \tfrac14 \|Z_{m}-Z_{n}\|^{2} = \frac{\|Z_{m}+c\|^{2}}{2}+\frac{\|Z_{n}+c\|^{2}}{2} - \left\|\frac{Z_{m}+Z_{n}}{2}+c\right\|^{2}=:b_{m,n}.
\]
Note that $a_{n}:=\|Z_{n}+c\|$ is decreasing, hence convergent to some $a_{*}\in\R$. As 
$(Z_{m}+Z_{n})/2\in \cZ_{m\wedge n}$, we have $a_{m,n}:=\|\frac{Z_{m}+Z_{n}}{2}+c\|\geq a_{m}\vee a_{n}$ and hence $\limsup_{m,n\to\infty} a_{m,n}\geq a_{*}$. It follows that $\limsup_{m,n\to\infty} b_{m,n}\leq 2a_{*}^{2}/2-a_{*}^{2}=0$, hence $(Z_{n})$ is a Cauchy sequence in $L^{2}(P)$. Its limit $Z$ must lie in $\cZ$ and then $\|Z+c\|=\lim_{n} \|Z_{n}+c\| \leq \|Z_{*}+c\|$ shows $Z=Z_{*}$, where the inequality is due to $\cZ_{n}\supset \cZ$.
\end{proof}

To pass to the limit of $f_{n}\oplus g_{n}$, we shall use the following result. 

\begin{lemma}\label{le:limitDecomp}
  Consider two sequences $f_{n}:\X\to\R$ and $g_{n}:\Y\to\R$, $n\geq1$ of measurable functions, and $\pi\in\cP(\mu,\nu)$. Suppose that
  \[
    \lim_{n\to\infty} f_{n}\oplus g_{n} = h \quad \pi\as
  \]
  for a measurable function $h:\X\times\Y\to\R$. Then there are functions $f:\X\to\R$ and $g:\Y\to\R$ such that $h=f\oplus g$ $\pi$-a.s. 
If $\pi\ll P$, the functions $f,g$ can be chosen to be measurable.
\end{lemma} 

The first part of \cref{le:limitDecomp} is \cite[Proposition~2.1]{RuschendorfThomsen.97}, the second is \cite[Proposition~3.19]{FollmerGantert.97}. The latter assumes $\pi\ll \pi_{\X}\otimes \pi_{\Y}$, where $\pi_{\X}, \pi_{\Y}$ are the marginals of~$\pi$, which is equivalent to $\pi\ll P$.\footnote{In fact, it suffices to assume that~$\pi$ is absolutely continuous wrt.\ \emph{any} product probability measure. See Step~4 in the proof of \cref{le:limitDecompSelf} below.} We comment briefly on the proof of \cref{le:limitDecomp} in the subsequent remark, but give a detailed proof of its (more difficult) symmetric version in \cref{le:limitDecompSelf} below.

\begin{remark}\label{rk:limitDecomp}
Results like \cref{le:limitDecomp} are surprisingly subtle; the difficulty is that the convergence of $f_{n}\oplus g_{n}$ only holds on a subset $E\subset \X\times\Y$. First of all, note that convergence of $f_{n}\oplus g_{n}$ does not imply a separate convergence of $f_{n}$ and $g_{n}$. If $E=\X\times\Y$, the limits of $f_{n}$ and $g_{n}$ do exist after a single normalization (e.g., $f_{n}(x_{0})=0$ for some fixed $x_{0}\in\X$) to pin down the familiar additive constant. %
But if~$E$ is sparse (as will typically be the case in our application), we have seen in \cref{se:discrete} that $f_{n}$ can be shifted by a different constant on each component of~$E$. While \cite{RuschendorfThomsen.97,FollmerGantert.97} do not use the concept of connectedness, their proofs roughly boil down to choosing a normalization for each component of~$E$. Because there are uncountably many components in general, the measurability of the resulting function is not guaranteed. In fact, a counterexample due to N.~Gantert (reported in \cite{RuschendorfThomsen.97}) shows that even when a limit $f\oplus g$ exists, it can happen that there is no  measurable choice of~$f$ and~$g$. (See also \cite{FollmerGantert.97} for further examples.) In the last part of \cref{le:limitDecomp}, the condition $\pi\ll P$ ensures that, after removing certain nullsets from~$\X$ and~$\Y$, only countably many normalizations are necessary. See the proof of \cref{le:limitDecompSelf} for details.
\end{remark}

We can now conclude the desired result on the shape of the optimal density $Z_{*}\in\cZ$, completing the proof of \cref{th:main}\,(i)--(iii).

\begin{proposition}\label{pr:approxLimit}
  There exist measurable $f:\X\to\R$ and $g:\Y\to\R$ such that
   $Z_{*} =  (f\oplus g - c)_{+}$ $P$-a.s.
\end{proposition} 

\begin{proof}
By \cref{le:approxSeq} we have $Z_{*} = \lim_{n\to\infty} (f_{n}\oplus g_{n} - c)_{+}$ $P$-a.s. for some bounded measurable functions $f_{n},g_{n}$. As $\pi_{*}$ is concentrated on $\{Z_{*}>0\}$, we also have
\begin{align*}
 Z_{*} = \lim_{n\to\infty} (f_{n}\oplus g_{n} - c) =  \left(\lim_{n\to\infty}  f_{n}\oplus g_{n} \right) - c \qquad \pi_{*}\as,
\end{align*} 
showing that the limit $h:=Z_{*}+c=\lim_{n\to\infty} f_{n}\oplus g_{n} $ exists and is finite $\pi_{*}$-a.s. Clearly $\pi_{*}\ll P$, hence
the condition of \cref{le:limitDecomp} is satisfied and the claim follows.
\end{proof}

\begin{proof}[Proof of \cref{th:main}\,{(i)--(iii)}]
As \cref{pr:approxLimit} verifies the assumption of \cref{le:weakDuality}\,(ii), the latter establishes the strong duality $\sP=\sD$, which is part (i) of \cref{th:main}. Part (ii) is shown in \cref{se:primalEx}. Moving on to part (iii), item (a) is precisely \cref{pr:approxLimit}, and it is trivial that (a) implies (b). Assuming (b), \cref{le:fgIntegrable} yields $(f,g)\in L^{1}(\mu)\times L^{1}(\nu)$ and then \cref{le:weakDuality}\,(ii) yields (c). Finally, \cref{le:weakDuality}\,(iii) shows that (c) implies (a). This shows the equivalence of (a), (b) and (c). The equivalence of (b) and (d) was already proved (after \cref{rk:regularizationParam}).
\end{proof} 

\section{Self-Transport: Proof of \cref{th:main}\hspace{.05em}(iv)}\label{se:proofMainSelf}

In this section we adapt the above arguments to the case of self-transport; i.e., the marginals coincide and the cost is symmetric: 
\begin{align}\label{eq:symm}
   (\X,\mu,\tmu)=(\Y,\nu,\tnu) \quad\qandq\quad c(x,y)=c(y,x).
\end{align}
Note that $P=\tmu\otimes\tmu$ is then also symmetric.

In this setting we may expect that there exist symmetric potentials $(f,g)$; i.e., with $f=g$. (Of course, not all potentials will be symmetric; see, e.g., \Cref{ex:discreteCounter}.) However, the proof of \cref{le:limitDecomp} above in general does not produce symmetric potentials, \emph{even if} the approximations $f_{n},g_{n}$ are symmetric. This is due to the normalizations for $f_{n}$ that are made in the proof. These normalizations are key to obtain convergence, but break the symmetry between~$f_{n}$ and~$g_{n}$. 

Below, we first argue that the approximations can indeed be chosen to be symmetric (\cref{le:finiteDimProblemSelf} and \cref{le:approxSeqSelf}). Then, we guarantee a symmetric limit with a precise construction that avoids normalizations on certain components and coordinates the normalizations between others (\cref{le:limitDecompSelf}).

\begin{lemma}\label{le:finiteDimProblemSelf}
  Let~\eqref{eq:symm} hold. Fix $n\in\N$ and bounded measurable functions $\phi_{i}: \X\to \R$, $1\leq i\leq n$ 
  with $\int \phi_{i}\,d\mu=0$. 
  Let
  \begin{multline*}
    \cZ_{n}^{\rm sym}=\bigg\{Z\in L^{2}(P): Z\geq0, \, Z(x,y)=Z(y,x), \\
     \int Z(x,y) \,P(dx,dy)=1, \, \int \phi_{i}(x)Z(x,y) \,P(dx,dy)=0,\, 1\leq i\leq n\bigg\}.   
  \end{multline*} 
  There is a unique solution $Z_{n}\in \cZ_{n}^{\rm sym}$ to
  \begin{align}\label{eq:primalfinitelyManyConstraintsSelf}
  \inf_{Z\in\cZ_{n}^{\rm sym}} \int cZ\,dP + \frac12 \|Z\|^{2}_{L^{2}(P)}
  \end{align}
  and $Z_{n}$ is characterized within $\cZ_{n}^{\rm sym}$ by being of the form
  \begin{equation}\label{eq:finitelyManyConstraintsFormSelf}
    Z_{n}(x,y)= (b_{0}+b\cdot \Phi(x) + b\cdot \Phi(y) -c(x,y))_{+}
  \end{equation}
  for some $b_{0}\in\R$ and $b\in\R^{n}$, where $\Phi:=(\phi_{1},\dots,\phi_{n})$.
\end{lemma}

\begin{proof}
  Note that $Z\in\cZ_{n}^{\rm sym}$ necessarily satisfies $\int \phi_{i}(y)Z(x,y) \,P(dx,dy)=0$ as this is implied by $\int \phi_{i}(x)Z(x,y) \,P(dx,dy)=0$ and $Z(x,y)=Z(y,x)$. Thus, following the same arguments as in the proof of \cref{le:finiteDimProblem}, it suffices to show that there exists $Z\in\cZ_{n}^{\rm sym}$ of the form~\eqref{eq:finitelyManyConstraintsFormSelf}. 
To that end, we now show that the symmetric problem
\begin{equation*}
  \inf_{b_{0},b\in\R^{n}} G(b_{0},b), \qquad G(b_{0},b):=\int (b_{0}+b\cdot\Phi(x)+b\cdot\Phi(y) - c(x,y))_{+}^{2}\,P(dx,dy)
\end{equation*}
admits a minimizer~$(b^{*}_{0},b^{*})$. As in the proof of \cref{le:finiteDimProblem}, we may assume that $b\cdot\Phi=0$ $\tmu$-a.s.\ only for $b=0$, and then existence of an optimizer $(b^{*}_{0},b^{*})$ follows by the same coercivity argument. Set $Z_{n}(x,y):=(b^{*}_{0}+b^{*}\cdot\Phi(x)+b^{*}\cdot\Phi(y) - c(x,y))_{+}$. The first-order condition at $b^{*}_{0}$ gives $\int Z_{n}(x,y)\,P(dx,dy)=1$ as before. The first-order condition at $b^{*}$ now gives
  $$
   \int \phi_{i}(x) Z_{n}(x,y)\,P(dx,dy)+ \int \phi_{i}(y) Z_{n}(x,y)\,P(dx,dy)=0, \quad 1\leq i \leq n.
  $$
  Because $Z_{n}$ and $P$ are symmetric, both integrals must have the same value; i.e., both vanish.  This shows that $Z_{n}\in \cZ_{n}$, as desired.
\end{proof} 

\begin{lemma}\label{le:approxSeqSelf}
  Let~\eqref{eq:symm} hold. Let $Z_{*}\in\cZ$ be the primal optimizer. There exist bounded measurable functions $f_{n}:\X\to\R$, $n\geq1$ such that
  \begin{align*}
    Z_{*} = \lim_{n\to\infty} (f_{n}\oplus f_{n} - c)_{+}\qquad\mbox{$P$-a.s.\ and in $L^{2}(P)$.}
  \end{align*} 
\end{lemma}

\begin{proof}
  Using \cref{le:finiteDimProblemSelf} instead of \cref{le:finiteDimProblem}, the argument is analogous to \cref{le:approxSeq}.
\end{proof}

The following passage to the limit $n\to\infty$ is the main step.

\begin{lemma}\label{le:limitDecompSelf}
  Let $(\X,\mu,\tmu)=(\Y,\nu,\tnu)$ and let $\pi\in\Pi(\mu,\mu)$ be symmetric; i.e., $\pi(dx,dy)=\pi(dy,dx)$. Consider a sequence $f_{n}:\X\to\R$, $n\geq1$ of measurable functions such that
  \[
    \lim_{n\to\infty} f_{n}\oplus f_{n} = h \quad \pi\as
  \]
  for a measurable function $h:\X\times\X\to\R$. Then there is a function $f:\X\to\R$ such that $h=f\oplus f$ $\pi$-a.s. 
  If $\pi\ll P$,
  then $f$ can be chosen to be measurable.
\end{lemma} 

\begin{proof}
  We give the proof in four steps. Steps~1 and~4 follow \cite{FollmerGantert.97, RuschendorfThomsen.97} whereas Steps~2 and~3 deal with the particular issue of constructing a symmetric limit. In a quite different context, issues with a similar flavor recently appeared in financial mathematics \cite{NutzWieselZhao.22a}.
  
  \emph{Step~1.} Consider the measurable (by~\cite[Lemma~2.1.7, p.\,107]{Bogachev.07volI}) set
  \begin{align*}
    S= \left\{\lim_{n\to\infty} (f_{n}\oplus f_{n})\mbox{ exists in $\R$} \right\}\subset\X\times\X.
  \end{align*} 
  We denote by $S_{x}=\{y\in\X:\,(x,y)\in S\}$ and $S^{y}=\{x\in\X:\,(x,y)\in S\}$ its sections, which are also measurable. For any $x,x'\in\X$, either
  \begin{align*}
    S_{x}=S_{x'} \quad\mbox{or}\quad S_{x}\cap S_{x'}=\emptyset.
  \end{align*} 
  Indeed, suppose that there is a point $z\in S_{x}\cap S_{x'}$, and consider any $y\in S_{x}$. Then $y\in S_{x'}$ as 
  \begin{align*}
    f_{n}(x')+f_{n}(y) = [f_{n}(x')+f_{n}(z)] - [f_{n}(x) + f_{n}(z)] + [f_{n}(x) + f_{n}(y)]
  \end{align*}
  and the terms on the right-hand side converge. 
  
  Fix a disintegration $\pi(dx,dy)=\mu(dx)\otimes\kappa(x,dy)$. %
  The set $\X_{0}=\{x\in\X:\,\kappa(x,S_{x})=1\}$ is measurable and has full $\mu$-measure. For $x\in\X_{0}$ we have in particular that $S_{x}\neq\emptyset$. Moreover, any $y\in\X_{0}$ is contained in $S_{x}$ for some $x\in\X_{0}$. This follows by the symmetry of $\pi$, as otherwise $\mu(S^{y})=0$. Define the equivalence relation $\sim$ on $\X_{0}$ via
  \begin{align*}
    x\sim x' \quad\mbox{if}\quad S_{x}=S_{x'}.
  \end{align*} 
  For any $x\in\X_{0}$, let $A(x)\subset\X_{0}$ be the equivalence class of $x$. This set is measurable as it has the representation $A(x)=\{x'\in\X_{0}:\, \lim (f_{n}(x')+f_{n}(y))\mbox{ exists in $\R$}\}$ for any $y\in S_{x}$. Let $(x_{\lambda})_{\lambda\in\Lambda}$ be a (possibly uncountable) system of representatives containing exactly one member of each equivalence class.
    
  \emph{Step~2.} Define $C(x):=A(x)\times (S_{x}\cap\X_{0})$ for $x\in\X_{0}$, and write $C_{\lambda}:=C(x_{\lambda})$. Then $(C_{\lambda})_{\lambda\in\Lambda}$ is a measurable partition of $S\cap(\X_{0}\times\X_{0})$. By definition, each $C_{\lambda}$ is a measurable rectangle, denoted 
  \[
  C_{\lambda}=A_{\lambda}\times B_{\lambda}.
  \]
  Both $(A_{\lambda})_{\lambda\in\Lambda}$ and $(B_{\lambda})_{\lambda\in\Lambda}$ are measurable partitions of $\X_{0}$. In fact, by symmetry, $\{A_{\lambda}:\,\lambda\in\Lambda\}=\{B_{\lambda}:\,\lambda\in\Lambda\}$. (Indeed, $y,y'\in\X_{0}$ belong to a common set $B_{\lambda}$ iff there exists $x\in\X_{0}$ such that $y,y'\in S_{x}$, which is in turn equivalent to  $y$ and $y'$ belonging to the same equivalence class. By the symmetry, the equivalence classes are precisely the sets $A_{\lambda}$, $\lambda\in\Lambda$.)
  In the language of \cref{de:connected}, $(C_{\lambda})_{\lambda\in\Lambda}$ are the connected components of $S\cap(\X_{0}\times\X_{0})$, seen as a subset of $\X_{0}\times\X_{0}$. The geometry is special here as any two connected points can be joined by a path with length $k=1$.
  
  Define the reflection $\hat C=\{(y,x): (x,y)\in C\}$ for any $C\subset\X\times\X$. By  symmetry, $\hat C_{\lambda}= C_{\lambda'}$ for some~$\lambda'$. We need to distinguish two types of components. On the one hand, we define 
  \begin{align*}
    \Lambda_{\rm diag}=\{\lambda\in\Lambda:\,C_{\lambda}=\hat C_{\lambda}\}.
  \end{align*}
  For $\lambda\in \Lambda_{\rm diag}$, we observe that $A_{\lambda}= B_{\lambda}$; i.e.,  $C_{\lambda}=A_{\lambda}\times A_{\lambda}$ is a square. 
  
  It is useful to visualize $S\cap(\X_{0}\times\X_{0})$ as a symmetric block matrix (\cref{fig:componentsSelf}). Then $C_{\lambda}$, $\lambda\in \Lambda_{\rm diag}$ are square blocks along the diagonal. Next, we describe the off-diagonal blocks, for which there is a symmetry between the lower and upper triangle matrices.
  
    \begin{figure}[htb]
    \centering
    \includegraphics[width=.3\textwidth]{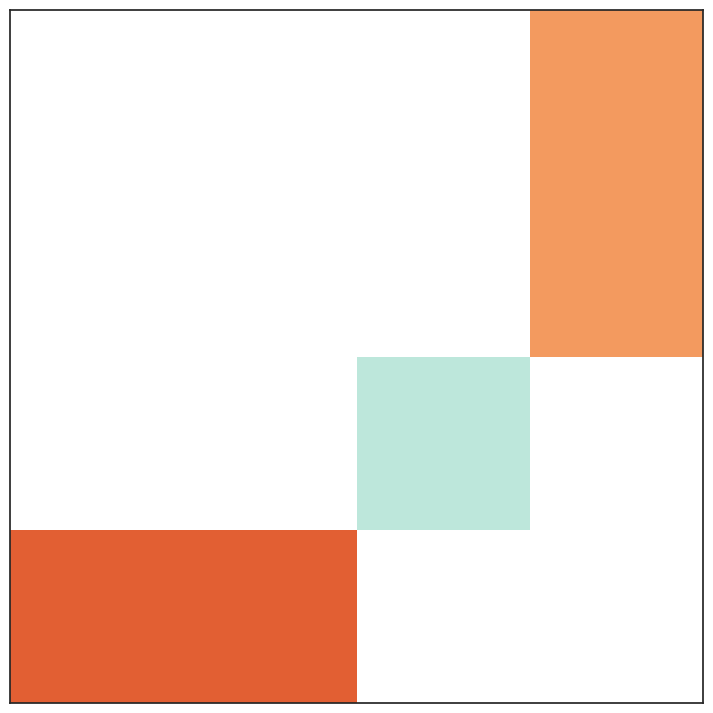}
    \caption{Example for $\X=[0,1]$ with three components (the diagonal is displayed top-left to bottom-right, as for matrices). Here $\Lambda_{\rm low}$, $\Lambda_{\rm diag}$, $\Lambda_{\rm upp}$ each have one element.}
    \label{fig:componentsSelf}
\end{figure}%

  Indeed, for each $\lambda\in \Lambda\setminus \Lambda_{\rm diag}$, there is exactly one $\lambda'\in \Lambda\setminus \Lambda_{\rm diag}$ such that $C_{\lambda'}=\hat C_{\lambda}$. To avoid redundancy, we partition $\Lambda\setminus \Lambda_{\rm diag}$ into $\Lambda_{\rm low}\cup\Lambda_{\rm upp}$ so that for each such unordered pair $\{\lambda,\lambda'\}$, one index (say $\lambda$) is in $\Lambda_{\rm low}$ and the other is in $\Lambda_{\rm upp}$.
 We note that the collections $(A_{\lambda})_{\lambda\in\Lambda_{\rm low}}$ and $(B_{\lambda})_{\lambda\in\Lambda_{\rm upp}}$ coincide. For the subsequent construction, it will be important that
 \begin{align}\label{eq:Apartition}
   (A_{\lambda})_{\lambda\in\Lambda_{\rm diag}} \cup (A_{\lambda})_{\lambda\in\Lambda_{\rm low}}\cup (B_{\lambda})_{\lambda\in\Lambda_{\rm low}} \qquad\mbox{is a partition of $\X_{0}$}.
 \end{align} 
  \emph{Step~3.} We define the function $f$ separately on each set of the partition~\eqref{eq:Apartition}.

  (i) Let $\lambda\in \Lambda_{\rm diag}$. As $C_{\lambda}=A_{\lambda}\times A_{\lambda}$ is a square, we see that $(x,y)\in C_{\lambda}$ implies $(x,x)\in C_{\lambda}$ (and similarly for $y$). The property $(x,x)\in C_{\lambda}\subset S$ means that $f_{n}(x) + f_{n}(x)$ is convergent, which of course means that $f_{n}(x)$ is convergent. In brief, $f_{n}(x)$ is convergent for all $x\in A_{\lambda}$. We thus define 
  \[
  f(x):=\lim_{n} f_{n}(x) \qforq x\in A_{\lambda}. %
  \]
  Clearly $f\oplus f= \lim_{n} (f_{n}\oplus f_{n})$ on $C_{\lambda}$.
  
  (ii) Let $\lambda\in \Lambda_{\rm low}$ and $(x,y)\in C_{\lambda}$.  In this case, $f_{n}(x_{\lambda})$ need not be convergent. Define 
  \begin{align}\label{eq:normaliz}
      f'_{n}(x):=f_{n}(x)-f_{n}(x_{\lambda}), \qquad f''_{n}(y):=f_{n}(y)+f_{n}(x_{\lambda}). 
  \end{align} 
  For $(x,y)\in C_{\lambda}$, we also have $(x_{\lambda},y)\in C_{\lambda}$, implying that $f''_{n}(y)$ and
  $
    f'_{n}(x)=[f_{n}(x)+f_{n}(y)] - [f_{n}(y)+f_{n}(x_{\lambda})]
  $ 
  are both convergent. Define 
  \begin{align*}
    f(x):=\lim_{n}f'_{n}(x) \qforq x\in A_{\lambda}, \qquad f(y):=\lim_{n}f''_{n}(y) \qforq y\in B_{\lambda}.
  \end{align*}
  These are well defined as $A_{\lambda}\cap B_{\lambda}=\emptyset$ for $\lambda\in \Lambda_{\rm low}$. Moreover,
  \begin{align*}
    f\oplus f=\lim_{n} (f'_{n}\oplus f''_{n})  = \lim_{n} (f_{n}\oplus f_{n}) \qonq C_{\lambda}.
  \end{align*}
  
  By~\eqref{eq:Apartition}, the combination of (i) and (ii) of Step~3 defines $f:\X_{0}\to\R$. Crucially, disjointness of the unions in~\eqref{eq:Apartition} ensures that there is no contradiction between our definitions of $f$ for different~$\lambda$ within (i) and (ii), and also not between (i) and (ii).
  
  We have  $f\oplus f= \lim_{n} (f_{n}\oplus f_{n})$ on $C_{\lambda}$ for $\lambda\in \Lambda_{\rm diag}\cup \Lambda_{\rm low}$. If $f\oplus f= \lim_{n} (f_{n}\oplus f_{n})$ on $C_{\lambda}$, the same holds on $\hat C_{\lambda}$. Thus, we also have $f\oplus f= \lim_{n} (f_{n}\oplus f_{n})$ on $C_{\lambda}$ for $\lambda\in\Lambda_{\rm upp}$. In summary, $f\oplus f= \lim_{n} (f_{n}\oplus f_{n})$ on $\X_{0}\times\X_{0}$. Finally, we set $f:=0$ on the $\mu$-nullset $\X\setminus\X_{0}$ and note that $\mu(\X_{0})=1$ implies $\pi(\X_{0}\times\X_{0})=1$ as $\pi\in\Pi(\mu,\mu)$. This gives the desired conclusion $f\oplus f= \lim_{n} (f_{n}\oplus f_{n})$ $\pi$-a.s. We remark that in general, the function $f$ need not be measurable, as it may incorporate \emph{uncountably} many normalizations~\eqref{eq:normaliz} with arbitrary choice of $x_{\lambda}$. The next step removes that issue.

    \emph{Step~4.}  Under the condition $\pi\ll P=\tmu\otimes\tmu$, we must have $\kappa(x,dy)\ll\tmu$ for $\mu$-a.e.\ $x\in\X$ by Fubini's theorem for kernels (or by \cref{rk:kernelWithoutDisintegration}),
 and we may choose $\kappa$ so that this holds without exceptional set. Then $\kappa(x_{\lambda},S_{x_{\lambda}})=1$ implies $\tmu(S_{x_{\lambda}})>0$. As the sets $S_{x_{\lambda}}$ are disjoint and $\tmu$ is a finite measure, $\tmu(S_{x_{\lambda}})>0$ can hold for at most countably many~$\lambda$. Thus there is a \emph{countable} set $\Lambda_{*}\subset\Lambda$ such that  $P(\cup_{\lambda\in\Lambda_{*}}C_{\lambda})=1$ and $P(C_{\lambda})>0$ for $\lambda\in\Lambda_{*}$. The set $\X_{1}:=\cup_{\lambda\in\Lambda_{*}}A_{\lambda}$ is measurable and satisfies $\tmu(\X_{1})=1$. We may redefine $f:=0$ outside $\X_{1}$ and still have $f\oplus f= \lim_{n} (f_{n}\oplus f_{n})$ $\pi$-a.s. This modified function $f$ is then a countable sum of the form $f=\sum_{\lambda\in\Lambda_{*}} f^{\lambda}\1_{D_{\lambda}}$ where $D_{\lambda}$ runs over elements of~\eqref{eq:Apartition} and each $f^{\lambda}$ is defined explicitly in~(i) or~(ii). In particular, $f$ is measurable.
\end{proof}

In the symmetric setting~\eqref{eq:symm}, the optimal coupling $\pi_{*}$ must be symmetric:  if $\pi_{*}(dx,dy)$ is an optimizer, then so is $\pi_{*}(dy,dx)$, hence both coincide, by uniqueness. Combining \cref{le:approxSeqSelf,le:limitDecompSelf}, we can now conclude as in \cref{pr:approxLimit}, completing the proof of \cref{th:main}. %

\begin{proposition}\label{pr:approxLimitSelf}
  Let~\eqref{eq:symm} hold. There exists a measurable $f:\X\to\R$ such that
  $Z_{*} =  (f\oplus f - c)_{+}$ $P$-a.s.
\end{proposition} 

\appendix

\section{Examples with Constant Cost}

If $c\equiv0$ and $(\mu,\nu)=(\tmu,\tnu)$, then clearly $\pi_{*}=P$. In particular, $\pi_{*}$ has full support. The next example illustrates that when $(\mu,\nu)\neq(\tmu,\tnu)$ are different (but still equivalent), the optimal coupling and even the optimal support can change.
  
\begin{example}[Reference can change optimal support]\label{ex:zeroCostDiscrete}
  Let $\X=\Y=\{0,1\}$ and $c\equiv0$.  Let
  \[
    \tmu=\tnu=\frac12(\delta_{0}+\delta_{1}), \qquad \mu=\nu=(1-\lambda)\delta_{0}+\lambda\delta_{1}, \qquad \lambda\in (0,1/4].
  \]
  We claim that the optimal density is 
  \begin{align*}
      Z_{*}:=(f\oplus f)_{+}, \qquad f(0) = 2-4\lambda, \quad f(1) = -2+8\lambda
  \end{align*} 
  and in particular that the optimal support is $\spt \pi_{*} =\{(0,0),(0,1),(1,0)\}$.
  
  By \cref{th:main}, it suffices to check that $(f\oplus f)_{+}$ is the $P$-density of a coupling. Indeed, $d\pi:=(f\oplus f)_{+}\,dP$ has weights
  \begin{align*}
    \pi\{0,0\}&=[f(0)+f(0)]/4 = 1-2\lambda,\\
    \pi\{0,1\}=\pi\{1,0\}&=[f(0)+f(1)]/4 = \lambda, \\
    \pi\{1,1\}&=[f(1)+f(1)]_{+}/4 = 0,
  \end{align*} 
  showing that $\pi\in\Pi(\mu,\nu)$.
\end{example}

Stated for general marginal spaces, the next proposition elaborates on \cref{ex:zeroCostDiscrete} by giving a sharp condition for $\pi_{*}$ to have (or not have) full support, or more precisely, to be equivalent to~$P$.

\begin{proposition}\label{pr:zeroCost}
  Let $c\equiv0$. 
Then
  \begin{align*}
    \pi_{*}\sim P \qquad\Longleftrightarrow\qquad \frac{d\mu}{d\tmu}  + \frac{d\nu}{d\tnu} >1\quad P\as
  \end{align*} 
   In that case, $d\pi_{*}/dP=\frac{d\mu}{d\tmu}  + \frac{d\nu}{d\tnu} -1$ and $(f,g)=(\frac{d\mu}{d\tmu} -\frac12,\frac{d\nu}{d\tnu} -\frac12)$ are potentials.
   
   In particular, if $\nu=\tnu$, the optimal coupling is $\pi_{*}=\mu\otimes\nu$ for any $\mu\sim\tmu$.
\end{proposition}
  
\begin{proof}
   Suppose that $\pi_{*}\sim P$. By \cref{th:main}, $0<Z_{*}=(f\oplus g -c)_{+}=f\oplus g$ $P$-a.s.\ for some potentials $f,g$. Then \eqref{eq:marginalEqnu} and~\eqref{eq:marginalEqmu} become
  $
    f=\frac{d\mu}{d\tmu}  - \int g\,d\tnu$ and $g=\frac{d\nu}{d\tnu}  - \int f\,d\tmu       $,
  which amounts to
  \begin{align*}
    f\oplus g = \frac{d\mu}{d\tmu}  \oplus \frac{d\nu}{d\tnu}  - 1 \quad P\as
  \end{align*}      
  In particular, $Z_{*}>0$ yields $\frac{d\mu}{d\tmu}  + \frac{d\nu}{d\tnu} >1$ $P$-a.s.
  
  Conversely, let $\frac{d\mu}{d\tmu}  + \frac{d\nu}{d\tnu} >1$ $P$-a.s.\ and define $(f,g):=(\frac{d\mu}{d\tmu}- \frac12,\frac{d\nu}{d\tnu}  - \frac12 )$. Going backwards, $(f,g)$ solve~\eqref{eq:marginalEqnu} and~\eqref{eq:marginalEqmu}, hence \cref{th:main} shows that $Z_{*}=(f\oplus g -c)_{+}=f\oplus g>0$. In particular, $\pi_{*}\sim P$.
  
  If $\nu=\tnu$, then $\frac{d\mu}{d\tmu}  + \frac{d\nu}{d\tnu} =\frac{d\mu}{d\tmu}+1>1$ $P$-a.s.\ and the claim follows.
\end{proof}

\section{Technical Remarks}

  For simplicity, we have assumed in the main text that the spaces $\X,\Y$ are Polish. In fact, we do not make direct use of the topology, and we can allow much more general spaces.

\begin{remark}[General spaces]\label{rk:minimalAssumptions}
  Our results on regularized optimal transport hold for arbitrary probability spaces $(\X,\cF_{\X},\mu)$ and $(\Y,\cF_{\Y},\nu)$ with countably generated $\sigma$-fields, not necessarily Polish. %
Indeed, countably generated $\sigma$-fields are sufficient for the finite-dimensional approximations in the proofs of \cref{le:approxSeq,le:approxSeqSelf}: If $\cF$ is countably generated, there exists a countable $\pi$-system $\cF_{0}=(A_{k})_{k\geq1}$ generating $\cF$ (meaning that $\cF_{0}$ is stable under pairwise intersection). Recall that any probability measure on $\cF$ is uniquely characterized by its restriction to $\cF_{0}$. The functions $\phi_{k}^{\mu}=\1_{A_{k}}-\mu(A_{k})$ then yield the desired sequence in the proofs of \cref{le:approxSeq,le:approxSeqSelf}.

The proofs of \cref{le:limitDecomp,le:limitDecompSelf} %
use the existence of disintegrations; i.e., regular conditional distributions. However, we only apply those results in the absolutely continuous case $\pi\ll P$, and then that existence holds on general spaces (\cref{rk:kernelWithoutDisintegration} below). 
\end{remark} 

The last remark recalls how to construct disintegrations from joint densities, without need for Polish or Blackwell/Souslin spaces (cf.\ \cite[Section~10.4]{Bogachev.07volII}).

\begin{remark}[Disintegration from density]\label{rk:kernelWithoutDisintegration}
  Let $\mu,\tmu\in\cP(\X)$ and $\nu,\tnu\in\cP(\Y)$ be probability measures on arbitrary measurable spaces $(\X,\cF_{\X})$ and  $(\Y,\cF_{\Y})$, and let $\pi\in\Pi(\mu,\nu)$ satisfy $\pi\ll \tmu\otimes\tnu$. Then there is a disintegration $\pi(dx,dy)=\mu(dx)\otimes \kappa(x,dy)$ with $\kappa(x,dy)\ll \tnu(dy)$ for all $x\in\X$. 
  
  This is a standard fact from probability theory. Indeed, note that $\pi\ll \tmu\otimes\tnu$ immediately implies $\mu\ll\tmu$. Let $D_{\mu}=\frac{d\mu}{d\tmu}$ and $D_{\pi}=\frac{d\pi}{d(\tmu\otimes\tnu)}$, and define $\kappa(x,dy):=\frac{D_{\pi}(x,y)}{D_{\mu}(x)}\1_{D_{\mu}(x)\neq0}\,\tnu(dy)$. Then $\kappa$ is a Markov kernel and for any $A\in\cF_{\X}$ and $B\in\cF_{\Y}$,
  \begin{align*}
    (\mu\otimes\kappa)(A\times B) = \int_{A} \left(\int_{B} \frac{D_{\pi}(x,y)}{D_{\mu}(x)}\1_{D_{\mu}(x)\neq0}\,\tnu(dy)\right)\,\mu(dx)=\pi(A\times B).
  \end{align*} 
\end{remark} 

\newcommand{\dummy}[1]{}

\end{document}